\documentclass[10pt,reqno]{amsart}
\usepackage[foot]{amsaddr}
\usepackage{etoolbox}
\usepackage[a4paper,margin=3cm]{geometry}
\usepackage{microtype}
\usepackage{graphicx}
\usepackage[justification=centering, margin=20pt]{caption}
\usepackage{subcaption}
\usepackage{booktabs}
\usepackage{dsfont}

\usepackage[colorlinks=true,citecolor=blue,linkcolor=blue]{hyperref}

\makeatletter
\renewcommand{\@secnumfont}{\bfseries}
\makeatother
\makeatletter
\def\subsection{\@startsection{subsection}{3}%
  \z@{.5\linespacing\@plus.7\linespacing}{.2\linespacing}%
  {\normalfont\bfseries}}
\makeatother

\renewcommand{\paragraph}[1]{\ \\ \textbf{#1}}

\usepackage[utf8]{inputenc}

\usepackage{amsthm}
\usepackage{amssymb}
\usepackage{amsfonts}

\usepackage[round]{natbib}

\newtheorem{theorem}{Theorem}
\newtheorem{lemma}{Lemma}

\usepackage{amsmath}
\usepackage{dsfont}

\usepackage{algorithm}
\usepackage{algorithmic}

\newcommand{\esp}[1]{\mathbb{E}\left[ #1 \right]}
\newcommand{\grando}[1]{O\left(#1 \right)}
\newcommand{\adfs}{ADFS}
\newcommand{\esdacd}{ESDACD}
\newcommand{\msda}{MSDA}

\newcommand{\D}{\mathcal{D}}

\title[ADFS]{Asynchronous Accelerated Proximal Stochastic Gradient for Strongly Convex Distributed Finite Sums}

\date{}

\author{
Hadrien Hendrikx $^\dagger$ 
\qquad Francis Bach \qquad Laurent Massouli\'e $^\dagger$
}

\address{INRIA - Département d’informatique de l’ENS \\
Ecole normale supérieure, CNRS, INRIA \\ 
PSL Research University, 75005 Paris, France}

\address{$\dagger$ Microsoft-INRIA Joint Centre}

\email{hadrien.hendrikx@inria.fr}
\email{francis.bach@inria.fr}
\email{laurent.massoulie@inria.fr}

\begin{document}

\begin{abstract}
In this work, we study the problem of minimizing the sum of strongly convex functions split over a network of $n$ nodes. We propose the decentralized and asynchronous algorithm ADFS to tackle the case when local functions are themselves finite sums with $m$ components. ADFS converges linearly when local functions are smooth, and matches the rates of the best known finite sum algorithms when executed on a single machine. On several machines, ADFS enjoys a $O (\sqrt{n})$ or $O(n)$ speed-up depending on the leading complexity term as long as the diameter of the network is not too big with respect to $m$. This also leads to a $\sqrt{m}$ speed-up over state-of-the-art distributed batch methods, which is the expected speed-up for finite sum algorithms. In terms of communication times and network parameters, ADFS scales as well as optimal distributed batch algorithms. As a side contribution, we give a generalized version of the accelerated proximal coordinate gradient algorithm using arbitrary sampling that we apply to a well-chosen dual problem to derive ADFS. Yet, ADFS uses primal proximal updates that only require solving one-dimensional problems for many standard machine learning applications. Finally, ADFS can be formulated for non-smooth objectives with equally good scaling properties. We illustrate the improvement of ADFS over state-of-the-art approaches with simulations.
\end{abstract}

\maketitle

{ \color{red} \fbox{\textbf{This is a preliminary version of the paper \url{https://arxiv.org/abs/1905.11394}}}}

\section{Introduction}

\begin{table*}[t]
\begin{center}
\begin{small}
\begin{sc}
\begin{tabular}{lcccc}
\toprule
Algorithm & Distributed & Asynchronous & Stochastic & Time \\
\midrule
Point-SAGA~\citep{defazio2016simple} & $\times$ & $\times$ & \checkmark & $nm + \sqrt{n m\kappa_s}$\\
AGD~\citep{nesterov2013introductory} & $\times$ & $\times$ & $\times$ & $nm\sqrt{\kappa_b}$\\
MSDA~\citep{scaman2017optimal} & \checkmark & $\times$ & $\times$ & $\sqrt{\kappa_b}\left(m + \frac{\tau}{\sqrt{\gamma}}\right)$\\
ESDACD~\citep{hendrikx2018accelerated} & \checkmark & \checkmark & $\times$ & $\left(m + \tau \right)\sqrt{\frac{\kappa_b}{\gamma}}$\\
DSBA~\citep{shen2018towards} & \checkmark & $\times$ & \checkmark & $\left(m + \kappa_s + \gamma^{-1}\right) \left(1 + \tau\right)$\\
\adfs~(this paper) & \checkmark & \checkmark & \checkmark & $m + \sqrt{m\kappa_s} + (1 + \tau)\sqrt{ \frac{\kappa_s}{\gamma}}$\\
\bottomrule
\end{tabular}
\end{sc}
\end{small}
\end{center}
\caption{Comparison of various state-of-the-art decentralized algorithms to reach accuracy $\varepsilon$ in regular graphs. Constant factors are omitted, as well as the $\log\left(\varepsilon^{-1}\right)$ factor in the \emph{\textsc{Time}} column. Reported runtimes for Point-SAGA and accelerated gradient descent (AGD) correspond to running them on a single machine with $n m$ samples. Times for ADFS and ESDACD are given for regular graphs.} 
\label{fig:table_speeds}
\vskip -0.1in
\end{table*}

The success of machine learning models is mainly due to their capacity to train on huge amounts of data. Distributed systems are the only way to process more data than one computer can store, but they can also be used to increase the pace at which models are trained by splitting the work among many computing nodes. Therefore, most machine learning optimization problems can be cast in the following way:
\begin{equation}
\label{eq:distributed_problem}
    \min_{\theta \in \mathbb{R}^d}\ \  \sum_{i=1}^n f_i(\theta),\ \ \ \ \ \ \  f_i(\theta) = \sum_{(x,y) \in \D_i} \ell(\theta, x, y),
\end{equation} 
where a loss function $\ell$ measures how far the prediction given by the parameter $\theta$ of size $d$ is from the true label~$y$ when given a sample $x$. We model the network by a graph of $n$ nodes, where the node $i$ has a local dataset $\D_i$ of size $m$, meaning that the dataset has $nm$ samples in total. We assume that each $f_i$ is $\sigma_i$-strongly convex and that each  component $\ell(\theta, x_j, y_j)$ coming from dataset $\D_i$ is $L_{j}$-smooth in the parameter~$\theta$ for a given training sample $(x_j, y_j)$. For more detailed definitions of smoothness and strong convexity, see, e.g.,~\citet{nesterov2013introductory,bubeck2015convex}. 

These problems are usually solved by first-order methods, and the basic distributed algorithms compute gradients in parallel over several machines~\citep{nedic2009distributed}. Another way to speed up training is to use \emph{stochastic} algorithms~\citep{bottou2010large}, that take advantage of the finite sum structure of the problem to use cheaper iterations while preserving fast convergence. This paper aims at bridging the gap between stochastic and distributed algorithms when local functions are smooth and strongly convex. The next paragraphs discuss the relevant state of the art for both distributed and stochastic methods, and Table~\ref{fig:table_speeds} sums up the speed of the main algorithms available to solve Problem~\eqref{eq:distributed_problem}. In the rest of this paper, following~\citet{scaman2017optimal}, we assume that processing one sample takes one unit of time, and that each communication takes time~$\tau$. The mixing time of the graph, denoted $\gamma^{-1}$, measures how well information spreads in the graph. It is a natural constant that appears in many distributed algorithms, and it can be be approximated as the inverse of the eigengap of the Laplacian matrix of the graph. Following notations from~\citet{xiao2017dscovr}, we define the batch and stochastic condition numbers $\kappa_b$ and $\kappa_s$ (which are classical quantities in the analysis of finite sum optimization) such that for all $i$, $\kappa_b \geq M_i / \sigma_i$ where $M_i$ is the smoothness constant of the function $f_i$ and $\kappa_s \geq \sum_{j=1}^m L_{j} / \sigma_i$. Generally, $\kappa_s = \grando{\kappa_b}$, leading to the practical superiority of stochastic algorithms.  

\paragraph{Distributed gradient methods.}
The aim of distributed algorithms is to obtain (when possible) a linear speedup, meaning that when run on $n$ nodes, the algorithm is $n$ times faster than the fastest single machine algorithm. In the standard centralized setting, all nodes can simply compute the gradient of their local functions and the global gradient is computed by a server as the sum of all local gradients. Provided the network is fast enough, this approach yields a linear speedup as long as there are less machines than samples~\citep{scaman2017optimal}. Instead of waiting for all nodes to send their gradients, asynchronous methods in which the server updates its parameter as soon as a worker finishes are generally used~\citep{recht2011hogwild, xiao2017dscovr}. Yet, computing gradients on older (or even inconsistent) versions of the parameter harms convergence~\citep{chen2016revisiting}. This effect combined with the communication bottleneck at the server generally causes centralized algorithms to only scale up to a certain point~\citep{leblond2016asaga, lian2017can}.

Thus, this paper focuses on decentralized methods that do not have this bottleneck~\citep{duchi2012dual}. In their synchronous versions, they alternate rounds of computations (in which all nodes compute gradients with respect to their local data) and communications in which nodes perform linear combinations of their freshly computed gradients with their neighbors~\citep{shi2015extra, nedic2017achieving, tang2018d}. Communication steps can thus be abstracted as multiplication by a so-called gossip matrix. MSDA~\citep{scaman2017optimal} is a batch decentralized synchronous algorithm, and it is optimal with respect to the constants $\gamma$ and $\kappa_b$ among algorithms that can only perform these two operations. As shown in Table~\ref{fig:table_speeds}, MSDA achieves a linear speedup compared to AGD as long as $\tau \gamma^{-\frac{1}{2}}$ is not too big compared to $m \kappa_b$. Yet, synchronous algorithms greatly suffer from something called the \emph{straggler} effect, meaning that they need to wait for the slowest node at each iteration. 

A way to mitigate this straggler effect is to use asynchronous decentralized algorithms inspired from randomized gossip algorithms~\citep{boyd2006randomized}, in which updates involve two nodes instead of the whole network \citep{nedic2009distributed, johansson2009randomized, colin2016gossip}.
Yet, asynchrony introduces variance and it is an open question whether all decentralized synchronous algorithms can have fast asynchronous equivalents. Recently, \esdacd~\citep{hendrikx2018accelerated} made a step in this direction by achieving the same rates as MSDA~\citep{scaman2017optimal}, its synchronous counterpart, but only when computations are faster than communications. Besides, ESDACD has a speedup of at best $n \gamma^{-\frac{1}{2}}$ over AGD~\citep{nesterov2013introductory}, an optimal single-machine batch methods. Therefore, it is mainly suited to graphs with a high connectivity.
Finally, it requires computing the gradients of the Fenchel conjugate of the full local functions, which are generally hard to get. 

\paragraph{Stochastic algorithms for finite sums.}
So far, we have only presented \emph{batch} methods that rely on computing \emph{full gradient} steps of each function $f_i$. Modern optimization methods generally process gradients on single data points as soon as they are computed. They can significantly outperform batch gradient methods when the dataset is large, as the fastest algorithms scale in $\sqrt{m\kappa_s}$ compared to $m \sqrt{\kappa_b}$ for batch methods. In the smooth and strongly convex setting, \emph{variance reduction} methods are needed to obtain stochastic algorithms converging linearly with rate $m + \kappa_s$, such as SAG~\citep{schmidt2017minimizing}, SDCA~\citep{shalev2013stochastic}, SVRG~\citep{johnson2013accelerating} and SAGA~\citep{defazio2014saga}. The $m + \sqrt{m\kappa_s}$ rate can then be obtained by using various acceleration techniques: Accelerated SDCA~\citep{shalev2014accelerated} was extended to the generic catalyst acceleration framework~\citep{lin2015universal}, whereas SAGA can be accelerated by simply using proximal steps instead of gradient steps~\citep{defazio2016simple}. APCG~\citep{lin2015accelerated} can be applied to a well-chosen dual formulation and Katyusha~\citep{allen2017katyusha} provides direct acceleration in the primal by using a negative momentum term.

\paragraph{Stochastic distributed algorithms.}
Batch computations can be easily parallelized, but stochastic algorithms gain their speed-up from processing one example at a time. Although it is often quite tempting to heavily distribute computations, care must be taken and it is not always the best choice to make. In particular, optimal batch distributed algorithms such as MSDA only improves Point-SAGA when $n > m$, i.e., when there are more nodes than samples per node. This is highly unlikely in a standard machine learning setup with massive data on a moderate size computing cluster, and highlights the need for new algorithms capable of efficiently leveraging extra computing power.  

As a matter of fact, and despite its sublinear convergence rate, efficiently distributing plain SGD in various settings is still an active line of work~\citep{lian2017asynchronous, mishchenko2018delay, tang2018d, olshevsky2018robust} for it is a both very simple and efficient algorithm. In the smooth and strongly convex setting, DSA~\citep{mokhtari2016dsa} and later DSBA~\citep{shen2018towards} obtained linearly converging stochastic distributed algorithms with more sophisticated approaches. Yet, they do not enjoy the $\sqrt{m\kappa_s}$ accelerated rates  and they need an excellent network with very fast communications since nodes communicate each time they process a sample, resulting in many communication steps.

As shown in Table~\ref{fig:table_speeds}, ADFS does not suffer from these drawbacks. It is very similar to APCG~\citep{lin2015accelerated} when executed on a single machine and thus gets the same $m + \sqrt{m\kappa_s}$ rate. Besides, this rate stays unchanged when the number of machines grows, meaning that it can process $n$ times more data in the same amount of time on a network of size $n$. This scaling lasts as long as $(1 + \tau)\sqrt{\kappa_s} \gamma^{-\frac{1}{2}} < m + \sqrt{m \kappa_s}$. This means that ADFS is at least as fast as MSDA unless both the network is extremely fast (communications are faster than a single gradient computation) and the diameter of the graph is very large compared to the size of the local finite sums. Therefore, \adfs~outperforms \msda~and DSBA in most standard machine learning settings, combining optimal network scaling with the efficient distribution of optimal sequential finite-sum algorithms.

Section~\ref{sec:model} details the model and the derivations to obtain a relevant dual problem. In Section~\ref{sec:alg}, we introduce a generalized version of APCG and then apply it to the dual problem in order to derive the generic ADFS algorithm. Section~\ref{sec:perfs} analyzes the speed of ADFS for a specific choice of parameters leading to the rate presented in Table~\ref{fig:table_speeds}. Finally, we illustrate the strengths of ADFS with a set of experiments covering various settings in Section~\ref{sec:experiments}.

\section{Model and derivations}
\label{sec:model}

\begin{figure}
  \centering
  \includegraphics[width=0.5\linewidth]{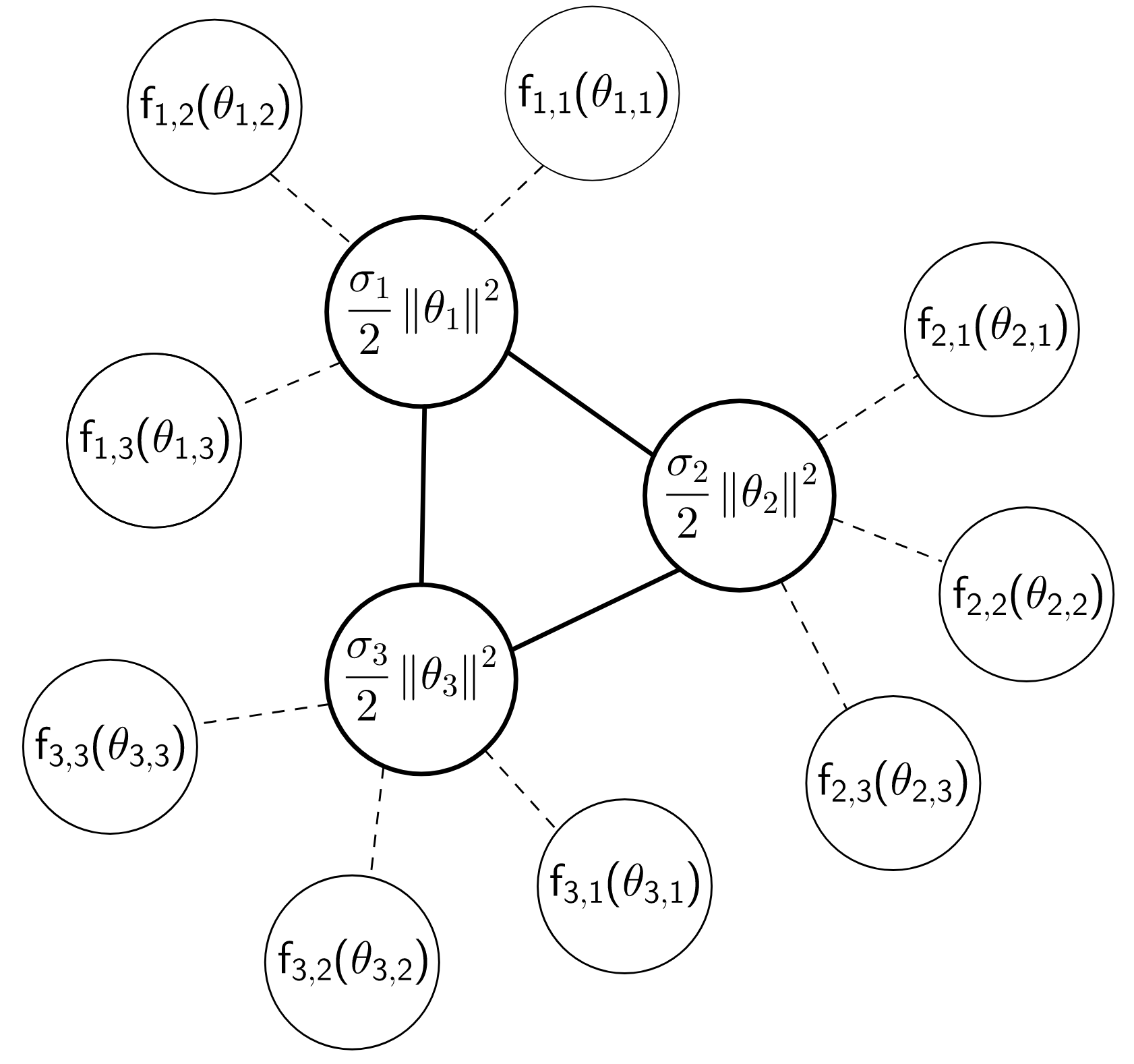}
  
\caption{Illustration of the augmented graph for $n=3$ and $m=3$. Thick lines are actual edges whereas dashed lines are virtual edges.}
\label{fig:extended_graph}
\end{figure}

We now specify our approach to solve the problem in Equation~\eqref{eq:distributed_problem}. Using the strong convexity assumption, we can rewrite all local functions in the form $f_i(x) = \sum_{j \in \D_i} f_{i,j}(x) + \frac{\sigma_i}{2}\|x\|^2$ for any $x \in \mathbb{R}^d$, with $f_{i,j}$ convex and $L_j$-smooth (we omit the $i$ in the notation), and where
$ \|x\|$ denotes the Euclidean norm.
 Then, \esdacd~and \msda~are obtained by applying accelerated (coordinate) gradient descent to the dual formulation of the following problem:
\begin{align*}
    &\min_{\theta \in \mathbb{R}^{n \times d}}\ \  \sum_{i=1}^n f_i(\theta_i) \ \  \text{ such that } \  \theta_i = \theta_j, \  \forall (i,j) \in E.
\end{align*}
In order to get a stochastic algorithm for finite sums, we consider a new virtual graph. The transformation is sketched in Figure~\ref{fig:extended_graph}, and basically consists in replacing each node by a star network. The centers of the stars are connected by the actual communication network, and the center of the star network replacing node $i$ has the local function $f^{\rm comm}_i(x) = \frac{\sigma_i}{2} \| x \|^2$. The center of node $i$ is then connected with $m$ nodes whose local functions are the functions $f_{i,j}$ for $j \in \D_i$. Updates in this augmented graph either involve two center nodes or the tip of a star and its center. Therefore, they never involve more than one function $f_{i,j}$. 
We denote by $E^{\rm comp}$ the set of virtual edges and $E^{\rm comm} = E$ the edges of the actual communication network. If we call the obtained virtual graph $G^+ = (E^+, V^+)$, updating an edge $(i,j)$ of $G^+$ consists in:
\begin{itemize}
\item Locally updating the parameter of node $i$ using only function $f_{i,j}$ if $(i,j) \in E^{\rm comp}$. This update takes time~$1$ and only involves the physical node $i$. If $(i,j) \in E^{\rm comp}$ then $i$ is the center of the star and $j$ the tip with function $f_{i,j}$.  We call $V_i$ the set of virtual nodes linked to node $i$.
    \item Performing a communication update between nodes $i$ and $j$ if $(i,j) \in E^{\rm comm}$. This update takes time $\tau$ and only involves nodes $i$ and $j$.  
\end{itemize}
Then, we can consider one parameter vector $\theta_{i,j}$ for each function $f_{i,j}$ and one vector $\theta_i$ for each
$f^{\rm comm}_i(x) = \frac{\sigma_i}{2} \| x \|^2$. Therefore, there is one parameter vector for each node in $V^+$. We impose the standard constraint that the parameter of each node must be equal to the parameters of its neighbors. The difference is that the neighbors are now taken in $G^+$ rather than $G$. This yields the following minimization problem:
\begin{align*}
    \begin{split}
        \min_{\theta \in \mathbb{R}^{|V^+| \times d}} \  &\sum_{i=1}^n \bigg[ \ \ \sum_{j \in V_i} f_{i,j}(\theta_{i,j}) + \frac{\sigma_i}{2} \| \theta_{i} \|^2 \bigg] \\
        \mbox{ such that } & {\theta_i = \theta_j \ \ \forall (i,j) \in E^{\rm comm}}\\
       & {\theta_{i,j} = \theta_i \ \ \forall j \in \{1, .., m\}}.
    \end{split}
\end{align*}
Noting $F(\theta) =  \sum_{i=1}^n \left[\sum_{j = 1}^{m} f_{i,j}(\theta_{i,j}) + \frac{\sigma_i}{2} \| \theta_{i} \|^2 \right]$, it is equivalent to:
\begin{equation}
\label{eq:primal_constrained}
    \min_{\theta \in \mathbb{R}^{|V^+| \times d}, \  A^T \theta = 0}  F(\theta),
\end{equation}
where $A \in \mathbb{R}^{|V^+| \times |E^+|}$ is such that for all $(i,j) \in E^+$, the column of $A$ corresponding to edge $(i,j)$ is equal to $\mu_{ij}(e_i - e_j)$ for some $\mu_{ij} > 0$. We note $e_{i}$ the unit vector of $\mathbb{R}^{|V^+|}$ corresponding to node $i$, and $e_{ij}$ for its virtual node $j$. Therefore, the dual formulation of this problem writes:
\begin{equation}
    \max_{\lambda \in \mathbb{R}^{|E^+| \times d}} - \sum_{i=1}^n \bigg[\sum_{j \in V_i} f_{i,j}^*(e_{ij}^TA\lambda) + \frac{1}{2\sigma_i} \| e_i^TA\lambda \|^2 \bigg],
\end{equation}
where the parameter $\lambda$ is the Lagrange multiplier associated with the constraints of Problem~\eqref{eq:primal_constrained}---more precisely, for an edge $(i,j)$, $\lambda_{ij} \in \mathbb{R}^d$ is the Lagrange multiplier associated with $\mu_{ij}( \theta_i -\theta_j)=0$. At this point, the functions $f_{i,j}$ are only assumed to be convex (and not necessarily strongly convex) meaning that the functions $f_{i,j}^*$ are potentially non-smooth. This problem could be bypassed by adding a strongly convex term to the tip of the stars before going to the dual formulation and subtracting it from the center. Yet, this approach fails when $m$ is large because the smoothness parameter of $f_{i,j}^*$ scales as $m / \sigma_i$ at best, whereas a smoothness of order $\sigma_i^{-1}$ is required to match optimal finite-sum methods. 

Instead, we will perform proximal updates on the $f_{i,j}^*$ terms. The rate of proximal gradient methods such as APCG does not improve if the non-smooth functions $f_{i,j}^*$ are strongly convex rather than simply convex.   Each $f_{i,j}^*$ is $(L_{j}^{-1})$-strongly convex (because $f_{i,j}$ was $L_{j}$-smooth), so we can rewrite the previous equation in order to put all the strong convexity in the quadratic term of the center node. We deduce from the form of $A$ that $e_{ij}^TA\lambda = - \mu_{ij} \lambda_{ij}$   when $(i,j)$ is a local edge. Therefore, if we define $\psi_{i,j}: x \mapsto f_{i,j}^*(- \mu_{ij} x) - \frac{\mu_{ij}^2}{2L_{j}}\|x\|^2$
and
\begin{equation}
    \label{eq:qa}
    q_A: x \mapsto \sum_{i=1}^n \bigg[\frac{1}{2\sigma_i} \| e_i^TAx \|^2 + \sum_{j \in V_i} \frac{1}{2L_{j}} \| e_{ij}^TAx \|^2\bigg], \!\!
\end{equation}
then $q_A$ can be rewritten as $q_A:x \rightarrow \frac{1}{2}x^T A^T\Sigma^{-1} A x$, where $\Sigma$ is the diagonal matrix such that $\Sigma_{ii} = \sigma_i$ if $i$ is a center node and $\Sigma_{ii} = L_{i}$ otherwise.  With these manipulations, the dual problem now writes:
\begin{equation}
\label{eq:dual_problem}
    \min_{\lambda \in \mathbb{R}^{|E^+| \times d}} q_A(\lambda) + \psi(\lambda),
\end{equation}
where $\psi(\lambda) = \sum_{i=1}^n \sum_{j = 1}^{m} \psi_{i,j}(\lambda_{ij})$.
The coordinate gradient of $q_A$ in the direction $(i,j)$ is equal to $\nabla_{ij}q_A: x \mapsto e_{ij}^T A^T \Sigma^{-1}Ax$, meaning that coordinate updates of the dual problem involve at most two nodes of the initial network. 

Therefore, using coordinate descent algorithms with a dual formulation from a well-chosen augmented graph allows us to handle both computations and communications in the same framework. Then, we can balance the ratio between communications and computations by simply adjusting the probability of picking a given kind of edge.

\section{Algorithm}
\label{sec:alg}
\subsection{Generalized APCG}
In this subsection, we study the generic problem:
\begin{equation}
\label{eq:generic_problem}
    \min_{x \in \mathbb{R}^d}\ \  f(x) + \sum_{i=1}^d \psi_i(x_i),
\end{equation}
where all the functions $\psi_i$ are convex and $f$ is $(\sigma_A)$-strongly convex in the norm $A^+A$ where $A^+A$ is an arbitrary projector, meaning that for all $x,y \in \mathbb{R}^d$: $$f(x) - f(y) \! \geq \! \nabla f(y)^T \! A^+A(x - y) + \textstyle \frac{\sigma_A}{2}(x - y)^T \! A^+A (x - y).$$ Besides, $f$ is assumed to be $(M_i)$-smooth in direction $i$ meaning that its gradient in the direction $i$ (noted $\nabla_i f$) is $(M_i)$-Lipschitz.
This is the general setting of the problem of Equation~\eqref{eq:dual_problem}, and proximal coordinate gradient algorithms are known to work well for these problems, which is why we would like to use APCG~\citep{lin2014accelerated}. Yet, we would like to pick different probabilities for computing and communication edges, whereas APCG only handles uniform coordinates sampling. Furthermore, the first term is strongly-convex only in the semi-norm $A^+A$, so APCG cannot be applied straightforwardly. Therefore, we introduce a more general version of APCG, presented in Algorithm~\ref{algo:generalized_apcg}, and we explicit its rate in Theorem~\ref{thm:gen_apcg}. This generalized APCG can then directly be applied to solve the problem of Equation~\eqref{eq:dual_problem}.

Let $R_i = e_i^T A^+A e_i$ and $p_i$ be the probability that coordinate~$i$ is picked to be updated.
$S$ is such that $S^2 \geq \frac{L_i R_i}{p_i^2}$ for all $i$. Then, following the approaches of~\citet{nesterov2017efficiency} and \citet{lin2015accelerated}, we define sequences $(a_t)$, $(A_t)$ and $(B_t)$ such that $  a_{t+1}^2 S^2 = A_{t+1} B_{t+1}$ where $B_{t+1} = B_t + \sigma_A a_{t+1}$ and $A_{t+1} = A_t + a_{t+1}$. From these, we define sequences $(\alpha_t)$ and $(\beta_t)$ such that $\alpha_t = \frac{a_{t+1}}{A_{t+1}}$ and  $\beta_t = \sigma_A \frac{ a_{t+1}}{B_{t+1}}$. Finally, we define the sequences $y_t$, $v_t$ and $x_t$ that are all initialized at $0$ and $(w_t)$ such that for all $t$, $w_t = (1 - \beta_t) v_t + \beta_t y_t$. We define $\eta_i = \frac{a_{t+1}}{B_{t+1} p_i}$ and the proximal operator:
\begin{equation*}
    {\rm prox}_{\eta_i \psi_i}: x \mapsto \arg\min_v \frac{1}{2\eta_i}\|v - x\|^2 + \psi_i(v).
\end{equation*}
\begin{algorithm}
\caption{Generalized APCG}
\label{algo:generalized_apcg}
\begin{algorithmic}
\STATE $y_0 = 0$, $v_0 = 0$, $t = 0$
\WHILE{$t < T$}
\STATE $y_t = \frac{(1 - \alpha_t) x_t + \alpha_t(1 - \beta_t)v_t}{1 - \alpha_t \beta_t}$
\STATE Sample $i$ with probability $p_i$
\STATE $z_{t+1} = z_{t+1} = (1 - \beta_t) v_t + \beta_t y_t - \eta_i \nabla_i f(y_t)$
\STATE $v_{t+1}^{(i)} = {\rm prox}_{\eta_i \psi_i}\left(z_{t+1}^{(i)}\right)$
\STATE $x_{t+1} = y_t + \frac{\alpha_t R_i}{p_i}(v_{t+1} - (1 - \beta_t) v_t - \beta_t y_t)$
\ENDWHILE
\end{algorithmic}
\end{algorithm}

For generalized APCG to work well, the proximal operator needs to be taken in norm $A^+A$, and so the non-smooth $\psi_i$ terms have to be separable in the norm $A^+A$. Since $A^+A$ is a projector, this constraint is equivalent to stating that either $R_i = 1$ (separability in the norm $A^+A$), or $\psi_i = 0$ (no proximal update to make). Therefore, we impose this natural constraint, which allows us to formulate the proximal update in standard squared norm since it is only used for coordinates $i$ for which $A^+A e_i = e_i$. This assumption is satisfied for our dual problem. Then, we can formulate Algorithm~\ref{algo:generalized_apcg} and analyze its rate, which is done by Theorem~\ref{thm:gen_apcg}.

\begin{theorem}
\label{thm:gen_apcg}
Let $F: x \mapsto f(x) + \sum_{i=1}^n \psi_i\left(x^{(i)}\right)$ such that all $\psi_i$ are convex and the function $f$ is $\sigma_A$-strongly-convex in the norm $A^+A$ and $(M_i)$-smooth in the direction $i$. If $1 - \beta_t - \frac{\alpha_t R_i}{p_i} > 0$ and $\psi_i = 0$ whenever $R_i = e_i^T A^+A e_i < 1$, the sequences $v_t$ and $x_t$ generated by APCG verify:
\begin{equation*}
    B_t\esp{\|v_t - \theta^*\|^2_{A^+A}} + 2 A_t \left[\esp{F(x_t)} - F(\theta^*)\right] \leq C_0,
\end{equation*}
where $C_0 = B_0 \|v_0 - \theta^*\|^2 + 2A_0\left[F(x_0) - F(\theta^*)\right]$ and $\theta^*$ is the minimizer of $F$. When  $\sigma_A > 0$, we can choose $\alpha_t = \beta_t = \rho$ and $A_t = \sigma_A^{-1} B_t = (1 - \rho)^{-t}$ with $\rho = \sqrt{\sigma_A}S^{-1}$.
\end{theorem}

The proof of Theorem~\ref{thm:gen_apcg}, as well as its guarantees in the non-strongly convex case are presented in Appendix~\ref{app:generalized_apcg}.

\subsection{ADFS}

\begin{algorithm}
\caption{ADFS}
\label{algo:sc_adfs}
\begin{algorithmic}
\STATE $x_0 = y_0 = v_0 = z_0 = 0^{(n + nm) \times d}$, $t = 0$
\WHILE{$t < T$}
\STATE $y_t = \frac{1}{1 + \rho}\left(x_t + \rho v_t\right)$
\STATE Sample $(i,j)$ with probability $p_{i,j}$
\STATE $v_{t+1} = z_{t+1} = (1 - \rho) v_t + \rho y_t - \tilde{\eta}_{ij} W_{ij}\Sigma^{-1}y_t$
\IF{$(i,j)$ is a computation edge}
\STATE $v_{t+1}^{(j)} = {\rm prox}_{\tilde{\eta}_{ij} \tilde{f}^*_{i,j}}\left(z_{t+1}^{(i,j)}\right)$
\STATE $v_{t+1}^{(i)} = z_{t+1}^{(i)} + z_{t+1}^{(i,j)} - v_{t+1}^{(i,j)}$
\ENDIF
\STATE $x_{t+1} = y_t + \frac{\rho R_{ij}}{p_{ij}}(v_{t+1} - (1 - \rho) v_t - \rho y_t)$
\ENDWHILE
\STATE \textbf{return} $\theta_t = \Sigma^{-1}v_t$
\end{algorithmic}
\end{algorithm}

We can now use Algorithm~\ref{algo:generalized_apcg} to solve the unconstrained optimization problem~\eqref{eq:dual_problem}. In the strongly convex setting, this yields the ADFS algorithm presented in Algorithm~\ref{algo:sc_adfs}, where  $\tilde{\eta}_{ij} = \frac{\rho \mu_{ij}^2}{\sigma_A p_i}$, $\rho$ is picked as in Theorem~\ref{thm:gen_apcg} and $W_{ij} \in \mathbb{R}^{n \times n}$ is the matrix such that $W_{ij} = (e_i - e_j)(e_i - e_j)^T$. Note that all the strong convexity has been relocated to the quadratic term. Proximal updates are therefore performed on $\tilde{f}^*_{i,j}: x \rightarrow f_{i,j}^*(x) - \frac{1}{2L_{j}}\|x\|^2$ rather than on $f_{i,j}^*$.
Going from generalized APCG to the distributed formulation actually requires several steps. In particular, it requires switching from edge variables to node variables by multiplying by $A$ on the left (and slightly more complicated manipulations to handle the proximal term). It also requires choosing the initial parameters of the sequences $A_t$ and $B_t$ so that $\alpha_t$ and $\beta_t$ are constant in order to have a simpler expression for the  algorithm. These manipulations are detailed in Appendix~\ref{app:algo_derivations}. Theorem~\ref{thm:rate_adfs} gives the rate of convergence of Algorithm~\ref{algo:sc_adfs}.

\begin{theorem}
\label{thm:rate_adfs}
Let $S$ be such that:
\begin{equation}
\label{eq:S}
    S^{-2} = \min_{ij} \frac{1}{\Sigma_i^{-1} + \Sigma_j^{-1}} \frac{p_{ij}^2}{\mu_{ij}^2 e_{ij}^T A^+A e_{ij}} ,
\end{equation}
and $\rho$ such that $\rho^2 = \lambda_{\min}^+ (A^T \Sigma^{-1} A ) S^{-2}$. Then $\theta_t$ and $x_t$ as output by Algorithm~\ref{algo:sc_adfs} verify:
\begin{align*}
\begin{split}
    c_1  &\esp{\|\theta_t - \theta^*\|^2} + 2 \left[\esp{F^*_A(A^+ x_t)} - F^*(\theta^*_A)\right]
     \leq \left[\sigma_A\|\theta^*_A\|^2 + 2\left(F^*_A(0) - F^*(\theta^*_A)\right)\right] (1 - \rho)^t,
\end{split}
\end{align*}
with $c_1 = \sigma_A\lambda_{\max}^+(A^T\Sigma^{-2}A)^{-1}$ and $F_A^* = q_A + \psi$. Note that $\theta^*$ is the minimizer of the primal function $F$ and $\theta^*_A$ is the minimizer of the dual function $F^*$.
\end{theorem}

We would like to insist again on the fact that the variables of Algorithm~\ref{algo:sc_adfs} are node variables in $\mathbb{R}^{(n + nm) \times d}$. They are obtained by multiplying the dual variables given by Algorithm~\ref{algo:generalized_apcg} by $A$ on the left. We recall that the matrix $\Sigma$ is a diagonal matrix defined at the end of Section~\eqref{sec:model}. The implementation of Algorithm~\ref{algo:sc_adfs} requires handling several technical details, that we present below.

\paragraph{Sparse updates.}
Although Algorithm~\ref{algo:sc_adfs} is written with full-dimensional updates, it is actually possible to implement it efficiently. Indeed, $W_{ij}$ is a very sparse matrix so only the following situations can happen:
\begin{itemize}
    \item If $(i,j) \in E^{\rm comp}$, then only the parameters of the center node $i$ and its $j$-th virtual node are updated, meaning that the update is purely local.
    \item If $(i,j) \in E^{\rm comm}$, then the update is simply a weighted difference of the parameters of nodes $i$ and $j$. Thus, it only requires nodes $i$ and $j$ to exchange parameters and sum them appropriately.
\end{itemize}
Note that the parameters of the nodes are only needed when they actually perform an update. Otherwise, $v_{t+1}$ and $y_{t+1}$ are obtained by simply mixing $v_t$ and $y_t$. Therefore, nodes can simply store how many updates have been done in total since their last update and perform them all at once before they perform the update. This is a distributed version of the efficient iterations introduced by~\citet{lee2013efficient}.

\paragraph{Primal proximal step.}
Algorithm~\ref{algo:sc_adfs} uses proximal steps performed on $\tilde{f}^*_{i,j}: x \rightarrow f_{i,j}^*(x) - \frac{1}{2L_{j}}\|x\|^2$ instead of $f_{i,j}$. Yet, ${\rm prox}_{\eta \tilde{f}^*_{i,j}}$ can be expressed directly from ${\rm prox}_{\eta f_{i,j}}$, which can be easily evaluated for many objective functions. The exact derivations are presented in Appendix~\ref{app:algo_derivations_primal_prox}. 

\paragraph{Linear case.}
For many standard machine learning problems, $f_{i,j}(\theta) = \ell(X_{i,j}^T\theta)$ with $X_{i,j} \in \mathbb{R}^d$. This implies that $f_{i,j}^*(\theta) = + \infty$ whenever $\theta \notin {\rm Vec}\left(X_{i,j}\right)$. Therefore, the proximal steps on the Fenchel conjugate only have support on $X_{i,j}$, meaning that the parameters of the local nodes only have support on~$X_{i,j}$. In this case, proximal updates are one-dimensional problems that can be solved very quickly using for example the Newton method when no analytical solution is available. Warm starts (initializing on the previous solution) can also be used for solving the local problems even faster so that in the end, the cost of performing a single one-dimensional proximal update is very small. Finally, storing only one scalar coefficient per virtual node is enough to implement the algorithm in this setting, thus greatly reducing the memory footprint of ADFS.

\paragraph{Shared schedule.}
Similarly to \citet{hendrikx2018accelerated}, we require all nodes to sample the same sequence of edges to update in order to implement Algorithm~\ref{algo:sc_adfs} even though they only actively take part in a small fraction of the updates. To generate this shared schedule, all nodes are given a seed and the sampling probabilities of all edges. This also allows them to precisely know how many contraction updates between $v_t$ and $y_t$ they need to perform before their next actual update.

\subsection{Non-smooth setting}
APCG can be applied to the problem of Equation~\eqref{eq:dual_problem} even if function $q_A$ is not strongly convex. Therefore, taking different values of $\alpha_t$, $\beta_t$, $A_t$ and $B_t$ leads to a formulation of ADFS that provides error guarantees even when primal functions $f_{i,j}$ are non-smooth.
\begin{theorem}
\label{thm:adfs_non_smooth}
If $F$ is non-smooth and $\tilde{\rho}^2 = \lambda_{\min}^+(A^T A) S^{-2}$ then NS-ADFS guarantees:
\begin{equation*}
 \mathbb{E}[F^*(Ax_t)] - F^*(A\theta^*) \leq 2(\tilde{\rho}t)^{-2}\|Av_0 - A\theta^*\|^2_{A^+A},
\end{equation*}
\end{theorem}
The guarantees provided by Theorem~\ref{thm:adfs_non_smooth} are weaker than in the smooth setting. In particular, we lose linear convergence and get the classical  accelerated sublinear $O(1/t^2)$ rate. We also lose the bound on the primal parameters--- recovering primal guarantees is beyond the scope of this work. The exact formulation of NS-ADFS as well as its derivation are presented in Appendix~\ref{app:algo_derivations}.

\section{Performances}
\label{sec:perfs}

Theorem~\ref{thm:rate_adfs} gives bounds on the expected error on the primal parameter after a given number of iterations. To assess the actual speed of the algorithm, we have to take into account that iterations can either be local or with a neighbor in the graph. These two types of iterations have different costs and take different times. Besides, increasing the number of communications causes nodes to wait more for their neighbors, thus decreasing the speed of the algorithm.

\subsection{Average time per iteration}
Executing the shared schedule mentioned in Section~\ref{sec:alg} means that some nodes may need to wait for an answer from their neighbors before they can perform local updates. This introduces a synchronization cost that we need to control. To do so, we bound the probability that a random schedule of fixed length exceeds a given execution time. Queuing theory~\citep{baccelli1992synchronization} is a tool of choice to obtain such bounds, and this problem can be cast as a fork-join queuing network with blocking~\citep{zeng2018throughput}. In particular, this theory tells us that the average time per step exists, so that synchrony does not cause the system to be slower and slower. Unfortunately, existing quantitative results are not precise enough for our purpose so we generalize the method introduced by~\citet{hendrikx2018accelerated} to get a bound on the average execution time. While their result is valid when the only possible operation is communication along edges, we extend it to the case in which nodes can also perform local computations. In particular, we define $p^{\max}_{\rm comm} = n \max_i \sum_{j \in \mathcal{N}(i)}p_{ij}$ where $\mathcal{N}(i)$ are the neighbors of node $i$ in the communication graph.
Then, the following theorem holds:

\begin{theorem}
\label{thm:synchronization_cost}
Let $T(t)$ be the time needed to execute a schedule of size $t$. If all nodes perform local computations with probability $p_{\rm comp} / n$ with $p_{\rm comp} > \bar{p}_{\rm comm}$ or if $\tau > 1$ then, $\mathbb{P}\left( T(t) \geq \nu t \right) \rightarrow 0$ as $t \rightarrow \infty$ with $C < 24$ and

\begin{equation}
    \nu = n^{-1}C\left(p_{\rm comp} + \tau p^{\max}_{\rm comm}\right).
\end{equation}
\end{theorem}

This novel result allows us to bound the execution time of the schedule, and therefore of ADFS. The assumption $p_{\rm comp} > p_{\rm comm}$ prevents \adfs~from benefiting from \emph{network acceleration} when local objectives are not finite sums ($m=1$) and communications are cheap ($\tau < 1$), and is responsible for the $1 + \tau$ factor instead of $\tau$ in Table~\ref{fig:table_speeds}. As a consequence, MSDA will turn out to be much faster in this regime. Yet, this is an actual restriction of following a schedule that is very intuitive in the limit of $p_{\rm comm} \rightarrow 1$ and $\tau$ arbitrarily small. Consider that one node (say node $0$) starts a local update at some point. Communications are very fast compared to computations so it is very likely that the neighbors of node $0$ will only perform communication updates, and they will do so until they have to perform one with node $0$. At this point, they will have to wait until node~$0$ finishes its local computation, which can take a long time. Now that the neighbors of node $0$ are also blocked waiting for the computation to finish, their neighbors will start establishing a dependence on them rather quickly. If the probability of computing is small enough and if the computing time is large enough, all nodes will sooner or later need to wait for node $0$ to finish its local update before they can continue with the execution of their part of the schedule. In the end, only node $0$ will actually be performing computations while all the others will be waiting. This phenomenon is not restricted to the limit case presented above and the synchronization cost blows up as soon as $p_{\rm comm} > p_{\rm comp}$ and $\tau < 1$. Yet, network operations generally suffer from communication protocols overhead whereas computing a single proximal update either has a closed-form solution or is a simple one-dimensional problem. Therefore, $\tau > 1$ is not a very restrictive assumption in the finite-sum setting.  

\subsection{Speed of \adfs}
We now prove the time to convergence of \adfs~presented in Table~\ref{fig:table_speeds}, and detail the conditions under which it holds. Indeed, Section~\ref{sec:alg} presents \adfs~in full generality but the different parameters have to be chosen carefully to reach optimal speed. In particular, we have to choose the coefficients $\mu$ to make sure that the graph augmentation trick does not cause the smallest positive eigenvalue of $A^T \Sigma^{-1} A$ to shrink too much. Similarly, $S$ is defined in Equation~\eqref{eq:S} by a minimum over all edges of a given quantity. This quantity heavily depends on whether the edge is an actual communication edge or a virtual edge. One can trade $p_{\rm comp}$ for $p_{\rm comm}$ so that the minimum is the same for both kind of edges, but Theorem~\ref{thm:synchronization_cost} tells us that this is only possible as long as $p_{\rm comp} > p_{\rm comm}$. 

More specifically, we define $L = AP^{\rm comm}A^T$ the Laplacian of the communication graph, where $P^{\rm comm}$ is the projection matrix on communication edges. Then, we define $\tilde{\gamma} = \max_{ij}\frac{\lambda_{\min}^+(L)n^2}{e_{ij}^TA^+Ae_{ij}|E|^2}$, $\sigma = \max_i \sigma_i$, $\kappa_i = \sigma_i^{-1}\sum_{j \in V_i} L_j$ and $\kappa_s = \max_i \kappa_i$. We choose the probabilities of virtual edges, such that $\sum_{j \in V_i}p_{ij}$ is constant for all $i$ and such that $p_{ij} \geq p_{\rm comp} \Delta_{\rm p} (1 + L_j \sigma_i^{-1})^{\frac{1}{2}} / (n S_{\rm comp})$ for some constant $\Delta_{\rm p}$ and $S_{\rm comp} = n^{-1}\sum_{i=1}^n\sum_{j \in V_i} (1 + L_j \sigma_i^{-1})^{\frac{1}{2}}$. When $(i,j)$ is a communication edge, we further assume that $p_{ij} \geq \Delta_{\rm p}  p_{\rm comm} / |E|$ and $p^{\max}_{\rm comm} \leq c_\tau p_{\rm comm}$ for some $c_\tau > 0$. Finally, we define $r_\kappa = \frac{1 + \min_i \kappa_i}{1 + \kappa_s}$.
\begin{theorem}
\label{thm:adfs_speed}
We choose $\mu_{ij} = \frac{1}{2}$ for communication edges, $\mu_{ij} = \frac{\lambda_{\min}^+(L)}{\sigma(1 + \kappa_i)}L_i$ for computation edges and $p_{\rm comm} = \min\Big(\frac{1}{2}, \left(1 + \sqrt{\frac{\tilde{\gamma}}{1 + \kappa_{min}}}S_{\rm comp}\right)^{-1}\Big)$. Then, running Algorithm~\ref{algo:sc_adfs} for $\rho^{-1}\log\left(\varepsilon^{-1}\right)$ steps guarantees an error less than $\varepsilon$ (in the sense of Theorem~\ref{thm:rate_adfs}), and the execution time $T$ is bounded by:
\begin{equation*}
    \frac{T}{\log\left(\frac{1}{\varepsilon}\right)} \leq \frac{\sqrt{2}C}{\Delta_{\rm p}}\left(\frac{m + \sqrt{m\kappa_s}}{\sqrt{r_\kappa}} + \left(1 + 2 c_\tau \tau\right)\sqrt{\frac{1 + \kappa_s}{\tilde{\gamma}}}  \right)
\end{equation*}
with probability $1$ as $\rho^{-1}\log\left(\varepsilon^{-1}\right) \rightarrow \infty$, where $C$ is the same as in Theorem~\ref{thm:synchronization_cost}.
\end{theorem}
The proof of Theorem~\ref{thm:adfs_speed} can be found in Appendix~\ref{app:algo_perfs}. When the graph is regular (the difference between the minimum and maximum degree nodes is bounded), $\tilde{\gamma} = \grando{\gamma}$.
Besides, we generally have $r_\kappa = \grando{1}$ and $\Delta_{\rm p} = \grando{1}$ when nodes sample their data from the same distribution and $m$ is large, yielding the rate of Table~\ref{fig:table_speeds}. Otherwise, $\sigma_i$ and sampling probabilities may be adapted to recover good guarantees, but it is beyond the scope of this paper. Note that taking computing probabilities exactly equal for all nodes is not necessary to ensure convergence.

\section{Experiments}
\label{sec:experiments}

\begin{figure*}
\centering
\begin{subfigure}{.33\textwidth}
  \centering
  \includegraphics[width=1.05\linewidth]{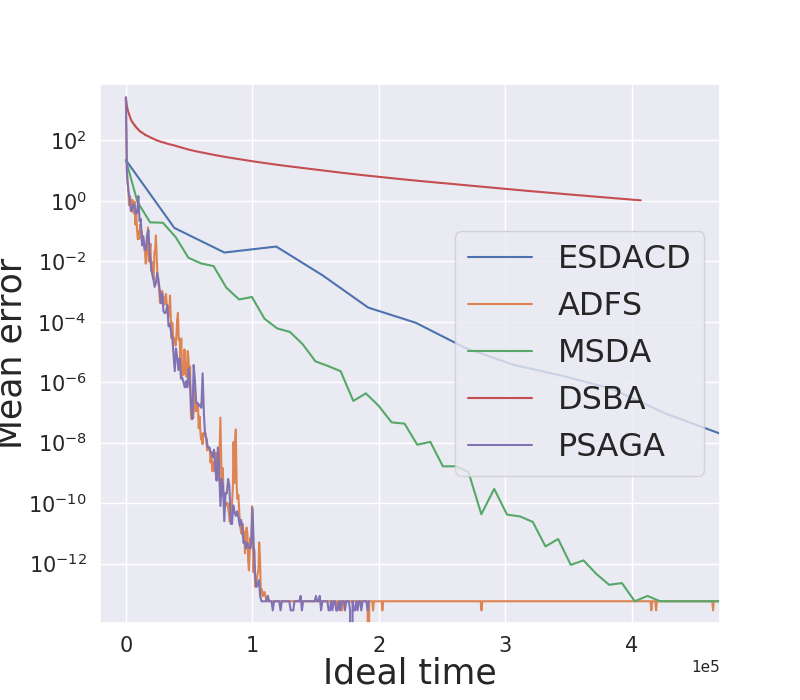}
  \vspace{-15pt}
\caption{Small network}
\label{fig:2x2_tau1_m1000}
\end{subfigure}%
\begin{subfigure}{.33\textwidth}
  \centering
  \includegraphics[width=1.05\linewidth]{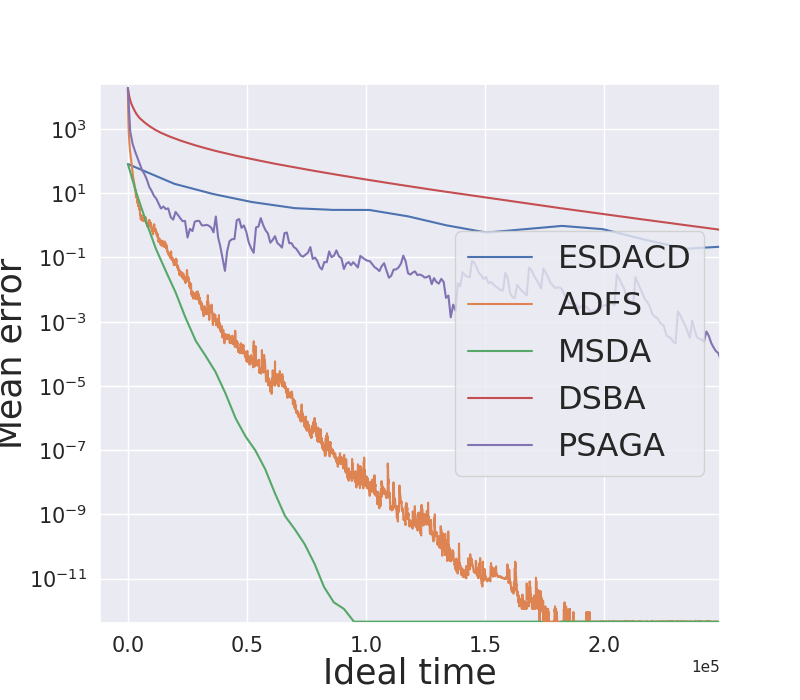}
  \vspace{-15pt}
  \caption{Large network}
\label{fig:10x10_m300}
\end{subfigure}
\begin{subfigure}{.33\textwidth}
  \centering
  \includegraphics[width=1.05\linewidth]{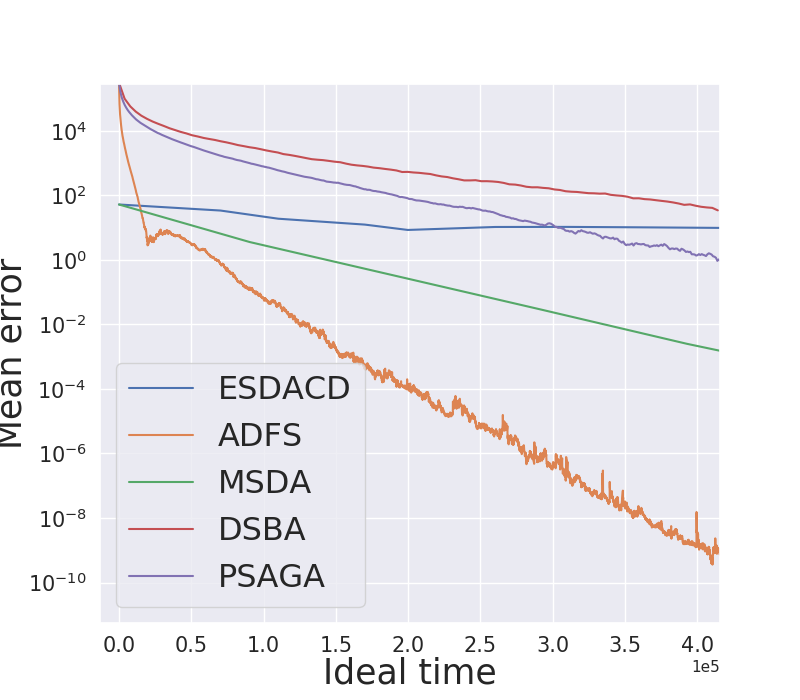}
  \vspace{-15pt}
  \caption{Large network, many samples per node}
\label{fig:7x7_tau5_m10000}
\end{subfigure}
\vspace*{-5pt}
\caption{Simulations with different network sizes and numbers of samples per node on the logistic regression task.}\vspace{-10pt}
\label{fig:test}
\end{figure*}

In this section, we illustrate the theoretical results by showing how ADFS compares with MSDA~\citep{scaman2017optimal}, the optimal distributed batch algorithm, ESDACD~\citep{hendrikx2018accelerated}, the batch version of ADFS, Point-SAGA~\citep{defazio2016simple}, a non-distributed accelerated proximal finite sum algorithm, and DSBA~\citep{shen2018towards}, a linearly converging synchronous stochastic decentralized algorithm. All algorithms were run with out-of-the-box parameters predicted by theory, except for DSBA for which the step size seemed quite conservative and was therefore multiplied by 4 (it was getting unstable after this point) to speed up convergence. We focus on a synthetic classification task in which each node $i$ has $m$ data points, denoted $X_{i,k} \in \mathbb{R}^d$ for $k \in \{1, \dots, m\}$ where $d=10$. Points are randomly drawn from a Gaussian distribution of variance $1$ centered at $-1$ for the first class and $1$ for the second class in a balanced way. The function at node~$i$ is $f_i(\theta) = \sum_{k=1}^m \log\left(1 + \exp(-y_{i,k} X_{i,k}^T\theta)\right) + \frac{\sigma_i}{2}\|\theta \|^2$. In this case, the proximal operator for one sample has no analytic solution but can be efficiently computed as the result of a one-dimensional optimization problem~\citep{shalev2013stochastic}.

In order to make comparisons easier, we assume that computing the proximal operator for $k$ samples is as fast as computing $k$ proximal operators on a single sample. Besides, we use an optimized implementation of MSDA and DSBA in which each node computes its gradient and sends it to its neighbors as soon as it has received all the gradients of the previous rounds. This avoids relying on global communication primitives and reduces nodes idle time. Experiments were run in a distributed manner on an actual computing cluster. Yet, plots are shown for \emph{idealized times} in order to abstract implementation details as well as ensure that reported timings are correct since there was no way to ensure that both the nodes and the network were not being occupied by other experiments. To emulate time, we draw delays from an exponential distribution with parameter 1 for computations, and multiply them by $\tau$ for communications. Note that nodes perform the schedule described in Section~\ref{sec:perfs} and start the next iteration as soon as they send their a gradient (non-blocking communications).

The first simulation, presented in Figure~\ref{fig:2x2_tau1_m1000} is run on a relatively fast ($\tau = 5$) and small ($2 \times 2$ grid) network, with $m=1000$ samples per nodes, and $\sigma_i = 1$. In this setting, the stochastic optimization speed-up clearly takes the lead on the distribution speed-up, so Point-SAGA beats all distributed algorithms. Yet, ADFS has a relatively low overhead so it matches the speed of Point-SAGA.  

For the second simulation, presented in Figure~\ref{fig:10x10_m300}, we run simulations on a much larger $10 \times 10$ grid but with only $300$ points per node and still $\sigma_i = 1$. In this setting, the distributed algorithms successfully leverage the extra computing power so that MSDA and ADFS beat Point-SAGA. Yet, the number of samples is not very big so MSDA is slightly faster than ADFS.

The last simulation runs on a $7 \times 7$ grid with $10^4$ samples per node and $\sigma_i = 50$, as shown in Figure~\ref{fig:7x7_tau5_m10000}. Even with such high values of $m$, MSDA is still very competitive and improves over Point-SAGA, thus showing that it can perform well beyond the $n > m$ limit suggested by the theory. Yet, ADFS gets a speed-up from both the high number of samples and the high number of computing nodes and outperforms all other methods.

Overall, DSBA and ESDACD do not perform so well in these experiments. For DSBA, we believe this is mainly due to the fact that it performs a communication step for each computation step, and we assumed communication time to be higher than computation time ($\tau=5$). Besides, it is the only non-accelerated methods. For ESDACD, it is due to the fact that batch computing times are significantly larger than communication times ($m > \tau$). 

\section{Conclusion}
In this paper, we provide a new stochastic algorithm for decentralized optimization. It leverages the finite-sum structure of the objective functions to match the rates of the best known sequential algorithms while having the network scaling of the optimal batch algorithms. It is particularly suited to the standard machine learning setting of composite optimization and also has a non-smooth formulation. 

The analysis in this paper could be extended to better handle more heterogeneous settings, both in terms of hardware (computing times, delays) and local functions (different regularities). Although considering finite sums drastically improved on this problem,
finding an asynchronous algorithm that can take advantage of arbitrarily low communication delays to remove the impact of the $\gamma$ factor on the rate is still an open question. Finally, we believe that the $\sqrt{n}$ distribution speed-up is optimal when the $\sqrt{m\kappa}$ term dominates. Proving this conjecture would be a nice addition to this work.

\bibliographystyle{abbrvnat}
\bibliography{biblio}

\newpage

\appendix 
\onecolumn

\section{Generalized APCG}
\label{app:generalized_apcg}

We study the following algorithm that starts from arbitrary $x_0, y_0, v_0$ and at each step picks a coordinate $i$ that will be updated. The algorithm is closely tied to APCG~\citep{lin2015accelerated}. We note $R_i = e_i^T A^+A e_i$. Then, we take the parameter sequences $\alpha_t, \beta_t, a_t, A_t \text{ and } B_t$ such that:

\begin{align*}
    & a_{t+1}^2 S^2 = A_{t+1} B_{t+1}, & B_{t+1} = B_t + \sigma a_{t+1}, & \ \ \ \ \ \ \ \ \ A_{t+1} = A_t + a_{t+1}, \\
    & \alpha_t = \frac{a_{t+1}}{A_{t+1}}, & \beta_t = \frac{\sigma a_{t+1}}{B_{t+1}}. & 
\end{align*}

Finally, we construct sequences $y_t, z_t, v_t \text{ and } x_t$ as in Algorithm~\ref{algo:generalized_apcg}, that we recall here:
\begin{equation}
    y_t = \frac{(1 - \alpha_t) x_t + \alpha_t(1 - \beta_t)v_t}{1 - \alpha_t \beta_t}
\end{equation}

\begin{equation}
    z_{t+1} = w_t - \frac{a_{t+1}}{B_{t+1} p_i}\nabla_i f_A(y_t)
\end{equation}
\begin{equation}
    v_{t+1} = z_{t+1}.
\end{equation}
If $\psi_i > 0$:
\begin{equation}
    \label{eq:vtplus1}
    v_{t+1}^{(i)} = \arg \min_{v}  V_i^t(v) \ \hat{=} \ \frac{B_{t+1} p_i}{2a_{t+1}}\|v - z_{t+1}^{(i)}\|^2  + \psi_i(v)
\end{equation}

\begin{equation}
\label{eq:xtplus1}
    x_{t+1} = y_t + \frac{\alpha_t R_i}{p_i}(v_{t+1} - w_t),
\end{equation}

where $w_t = (1 - \beta_t) v_t + \beta_t y_t$ and $S$ is such that $ \frac{M_i R_i}{p_i^2} \leq S^2$ for all $i$. Note that $v_{t+1}^{(j)} = w_t^{(j)}$ for $j \neq i$. We can then prove Theorem~\ref{thm:gen_apcg}. To do so, we define $\tilde{v}_{t+1}$ such that $\tilde{v}_{t+1}^{(i)} = e_i^T \arg \min_{v} V_i^t(v)$ for all $i$. Then, we give the following lemma, that we will prove later:

\begin{lemma}
\label{lemma:lyapunov_psi}
If $1 - \beta_t - \frac{\alpha_t}{p_i} \geq 0$ then for any $t$, we can write $x_t = \sum_{l=0}^t \delta^{(t)}_l v_l$ such that $\sum_{l=0}^t \delta^{(t)}_l = 1$ and for any $l$, $\delta^{(t)}_l > 0$. We define $\hat{\psi}_t = \sum_{l=0}^t \delta^{(t)}_l \psi(v_l)$. Then, if $R_i = 1$ whenever $\psi_i \neq 0$, $\psi(x_t) \leq \hat{\psi}_t$ and:
\begin{equation}
    \mathbb{E}_{i_t}\left[\hat{\psi}_{t+1}\right] \leq \alpha_t \psi(\tilde{v}_{t+1}) + (1 - \alpha_t)\hat{\psi}_t.
\end{equation}
\end{lemma}

Note that Lemma~\ref{lemma:lyapunov_psi} is a small generalization to arbitrary sampling probabilities of the beginning of the proof in~\citep{lin2015accelerated}. We can now prove the main theorem.

\begin{proof}
The goal of this proof is to follow~\citet{nesterov2017efficiency}. To achieve this, we need to expand $\|v_{t+1} - \theta^*\|^2$. In the original proof, $v_{t+1} = w_t - g$ where $g$ is a gradient term so the expansion is rather straightforward. In our case, $v_{t+1}$ is defined by a proximal mapping so a bit more work is required. We start by showing the following equality:
\begin{align}
\label{eq:sc_prox}
\begin{split}
    \frac{B_{t+1}p_i}{2a_{t+1}} [\|v_{t+1} - \theta^*\|^2_{A^+A} + &\|v_{t+1} - w_t\|^2_{A^+A} - \|\theta^* - w_t\|^2_{A^+A}] \\
    &\leq \langle \nabla_i f_A(y_t), \theta^* - v_{t+1} \rangle_{A^+A} + \psi_i\left({\theta^*}^{(i)}\right) - \psi_i\left(v_{t+1}^{(i)}\right).
    \end{split}
\end{align}

When $\psi_i = 0$, it follows from using $v_{t+1} = w_t - \frac{a_{t+1}}{B_{t+1} p_i}\nabla_i f_A(y_t)$ and basic algebra (expanding the squared terms).

When $\psi_i \neq 0$, $A^+Ae_i = e_i$ because $e_i^TA^+Ae_i = 1$ and $A^+A$ is a projector. Therefore, we obtain 
\begin{equation}
\label{eq:vt1proj}
    \|v_{t+1} - \theta^*\|^2_{A^+A} - \|w_t - \theta^*\|^2_{A^+A} = \|v_{t+1}^{(i)} - {\theta^*}^{(i)}\|^2 - \|w_t^{(i)} - {\theta^*}^{(i)}\|^2
\end{equation}
because $v_{t+1}$ is equal to $w_t$ for coordinates other than $i$. We now use the strong convexity of $V_i^t$ at points $v_{t+1}^{(i)}$ (its minimizer) and ${\theta^*}^{(i)}$ ($i$-th coordinate of a minimizer of $F$) to write that $V_i^t(v_{t+1}^{(i)}) + \frac{B_{t+1}p_i}{2a_{t+1}}\|v_{t+1}^{(i)} - {\theta^*}^{(i)}\|^2 \leq V_i^t({\theta^*}^{(i)})$. This is a key step from the proof of~\citep{lin2015accelerated}. Then, expanding the $V_i^t$ terms yields:

\begin{align*}
    \frac{B_{t+1}p_i}{2a_{t+1}} &\left[\|v_{t+1}^{(i)} - {\theta^*}^{(i)} \|^2 + \|v_{t+1}^{(i)} - w_t^{(i)} + \frac{a_{t+1}}{B_{t+1} p_i}\nabla_i f_A(y_t)\|^2 - \|\theta^* - w_t + \frac{a_{t+1}}{B_{t+1} p_i}\nabla_i f_A(y_t)\|^2 \right]\\
    &\leq  \psi_i\left({\theta^*}^{(i)}\right) - \psi_i\left(v_{t+1}^{(i)}\right).
\end{align*}
We can now retrieve Equation~\eqref{eq:sc_prox} by pulling gradient terms out of the squares and using Equation~\eqref{eq:vt1proj}. We now evaluate each term of Equation~\eqref{eq:sc_prox}. First of all, we use Equation~\ref{eq:xtplus1} and the fact that $w_t - v_{t+1} = e_i^T(w_t - v_{t+1})$ to show:
\begin{align*}
    \esp{\frac{a_{t+1}}{p_i}\langle \nabla_i f_A(y_t), \theta^* - v_{t+1} \rangle_{A^+A}} &= a_{t+1} \esp{\langle \frac{1}{p_i}\nabla_i f_A(y_t), \theta^* - w_t\rangle_{A^+A}} + A_{t+1}\esp{\langle \nabla_i f_A(y_t), \frac{\alpha_t}{p_i}(w_t - v_{t+1}) \rangle_{A^+A}} \\
    &= a_{t+1} \langle \nabla f_A(y_t), \theta^* - w_t \rangle + A_{t+1}\esp{\langle \nabla_i f_A(y_t), y_t - x_{t+1} \rangle},
\end{align*}

where we used that $R_i = e_i^TA^+Ae_i$ and $y_t - x_{t+1} = \frac{\alpha_t R_i}{p_i}(w_t - v_{t+1})$. The rest of this proof now follows the analysis from~\citet{nesterov2017efficiency}. For the first term, we use the strong convexity of $f$ as well as the fact that $w_t = y_t - \frac{1 - \alpha_t}{\alpha_t} (x_t - y_t)$ to obtain:

\begin{align*}
a_{t+1}& \nabla f_A(y_t)^T (\theta^* - w_t) = a_{t+1} \nabla f_A(y_t)^T \left(\theta^* - y_t + \frac{1 - \alpha_t}{\alpha_t}(x_t - y_t)\right)\\
&\leq a_{t+1} \left(f_A(\theta^*) - f_A(y_t) - \frac{1}{2} \sigma \|y_t - \theta^*\|^2_{A^+A} + \frac{1 - \alpha_t}{\alpha_t}(f_A(x_t) - f_A(y_t)) \right) \\
&\leq a_{t+1} f_A(\theta^*) - A_{t+1} f_A(y_t) + A_t f_A(x_t) - \frac{1}{2} a_{t+1} \sigma \|y_t - \theta^*\|^2_{A^+A}.
\end{align*}

For the second term, we use the fact that $x_{t+1} - y_t$ has support on $e_i$ only (just like $v_{t+1} - w_t$) and the directional smoothness of $f_A$ to obtain:

\begin{align*}
    A_{t+1} \langle \nabla_i f_A(y_t), y_t - x_{t+1} \rangle &\leq A_{t+1} \left[f_A(y_t) - f_A(x_{t+1}) + \frac{M_i}{2}\|x_{t+1} - y_t\|^2\right] \\
    &\leq A_{t+1} \left(f_A(y_t) - f_A(x_{t+1})\right) +\frac{B_{t+1}}{2} \frac{L_i R_i}{p_i^2} \frac{a_{t+1}^2}{A_{t+1}B_{t+1}} R_i\|e_i^T (v_{t+1} - w_t)\|^2\\
    &\leq A_{t+1} \left(f_A(y_t) - f_A(x_{t+1})\right) +\frac{B_{t+1}}{2} \|v_{t+1} - w_t\|^2_{A^+A}
\end{align*}

Noting $\Delta f(x) = \esp{f(x)} - f(\theta^*)$ and remarking that $a_{t+1} = A_{t+1} - A_t$, we obtain, using that $\alpha_t = \frac{a_{t+1}}{A_{t+1}}$: 

\begin{align*}
    \esp{\frac{a_{t+1}}{p_i}\langle \nabla_i f_A(y_t), \theta^* - v_{t+1} \rangle_{A^+A}} \leq A_t \Delta f_A(x_t)& - A_{t+1} \Delta f_A(x_{t+1}) + \frac{B_{t+1}}{2}\esp{ \|w_t - v_{t+1}\|^2_{A^+A}}\\
    &- \frac{a_{t+1}\sigma}{2}\|y_t - \theta^*\|^2_{A^+A}.
\end{align*}

Using Lemma~\ref{lemma:lyapunov_psi}, we derive in the same way:

\begin{align*}
    \esp{\frac{a_{t+1}}{p_i} \left[\psi_i\left({\theta^*}^{(i)}\right) - \psi_i\left(v_{t+1}^{(i)}\right)\right]} &= a_{t+1}\psi(\theta^*) - A_{t+1} \alpha_t  \psi(\tilde{v}_{t+1})\\
    &\leq A_t \left(\hat{\psi}_t - \psi(\theta^*)\right) - A_{t+1}\left(\hat{\psi}_{t+1} - \psi(\theta^*) \right).
\end{align*}

Now, we can multiply Equation~\ref{eq:sc_prox} by $\frac{a_{t+1}}{p_i}$ and take the expectation over $i$. The $\|v_{t+1} - w_t\|^2_{A^+A}$ terms cancel and we obtain:

\begin{align*}
\frac{B_{t+1}}{2}\esp{\|v_{t+1} - \theta^*\|^2_{A^+A}} + A_{t+1} \Delta \hat{F}_A(x_{t+1}) \leq A_t \Delta \hat{F}_A(x_t) + \frac{B_{t+1}}{2}\|w_t - \theta^*\|^2_{A^+A} - \frac{a_{t+1}\sigma}{2}\|y_t - \theta^*\|^2_{A^+A},
\end{align*}

where $\Delta \hat{F}_A(x_t) = \Delta f_A(x_t) +  \esp{\hat{\psi}_t} - \psi(\theta^*)$. convexity of the squared norm yields $\|w_t - \theta^*\|^2_{A^+A} \leq (1 - \beta_t)\|v_t - \theta^*\|^2_{A^+A} + \beta_t \|y_t - \theta^*\|^2_{A^+A}$. Now remarking that $B_{t+1}(1 - \beta_t) = B_t$ and $a_{t+1}\sigma = B_{t+1} \beta_t$, and summing the inequalities until $t=0$, we obtain:

\begin{align*}
B_{t}\|v_{t} - \theta^*\|^2_{A^+A} + 2A_{t} \Delta \hat{F}_A(x_{t}) \leq 2A_0 \Delta F_A(x_0) + B_0 \|v_0 - \theta^*\|^2_{A^+A}.
\end{align*}

We finish the proof by using the fact that $\psi(x_t) \leq \hat{\psi}_t$. The growth of the coefficients $A_t$ and $B_t$ can then be proven by an easy induction, depending on whether $\sigma = 0$ or not.
\end{proof} 

Now, we prove Lemma~\ref{lemma:lyapunov_psi}:

\begin{proof}
This lemma is a generalization of the lemma from APCG with arbitrary probabilities (instead of just uniform ones). The proof is very similar. We start the proof by expressing $x_{t+1}$ in terms of $x_t$, $v_{t+1}$ and $v_t$:

Then,
\begin{align*}
    x_{t+1} &= y_t + \frac{\alpha_t R_i}{p_i}(v_{t+1} - w_t)\\
    &= \frac{\alpha_t R_i}{p_i}v_{t+1} + \left(1 - \frac{\alpha_t \beta_t R_i}{p_i}\right) y_t  - \frac{\alpha_t(1 - \beta_t)R_i}{p_i}v_{t}\\
    &= \frac{\alpha_t}{p_i}v_{t+1} + \left(1 - \frac{\alpha_t \beta_t R_i}{p_i}\right)\frac{(1 - \alpha_t) x_t + \alpha_t(1 - \beta_t)v_t}{1 - \alpha_t \beta_t}  - \frac{\alpha_t(1 - \beta_t)R_i}{p_i}v_{t}\\
    &= \frac{\alpha_t R_i}{p_i}v_{t+1} + \alpha_t(1 - \beta_t)\left[\frac{1 - \frac{\alpha_t \beta_t R_i}{p_i}}{1 - \alpha_t \beta_t} - \frac{R_i}{p_i} \right]v_t + \left(1 - \frac{\alpha_t \beta_t R_i}{p_i}\right)\frac{(1 - \alpha_t)}{1 - \alpha_t \beta_t}\sum_{l=0}^t \delta_l^{(t)} v_l\\
    &= \frac{\alpha_t R_i}{p_i}v_{t+1} + \frac{\alpha_t(1 - \beta_t)}{1 - \alpha_t \beta_t}\left(1 - \frac{R_i}{p_i}\right)v_t + \left(1 - \frac{\alpha_t \beta_t R_i}{ p_i}\right)\frac{(1 - \alpha_t)}{1 - \alpha_t \beta_t} x_t . 
\end{align*}

At this point, we know that all coefficients sum to 1. Indeed, they all sum to 1 at the first line and all we have done is express $w_t$ and then $y_t$ as convex combinations of other terms, thus keeping the value of the sum unchanged. Yet, $p_i < 1$ so the coefficient on the second term is negative. Therefore, we now show by recursion that for $t \geq 1$

\begin{equation}
\label{eq:xt_convex}
    x_t = \frac{\alpha_t R_i}{p_i}v_t + \sum_{l=0}^{t-1} \delta^{(t)}_l v_l,
\end{equation}

This implies that if we expand the $x_t$ term then we get a positive coefficient on $v_t$. For $t=0$, $x_0 = v_0$ and so $x_1 = \frac{\alpha_0 R_i}{p_i}v_1 + \left( 1 - \frac{\alpha_0 R_i}{p_i}\right)v_0$ because the sum is of the coefficients is equal to 1.

We now assume that Equation~\ref{eq:xt_convex} holds for a given $t > 0$.

\begin{align*}
    \delta_{t}^{(t+1)} = &\frac{\alpha_t(1 - \beta_t)}{1 - \alpha_t \beta_t}\left(1 - \frac{R_i}{p_i}\right) + \frac{\alpha_t R_i}{p_i} \left(1 - \frac{\alpha_t \beta_t R_i}{p_i}\right)\frac{(1 - \alpha_t)}{1 - \alpha_t \beta_t} \\
    &= \frac{\alpha_t}{1 - \alpha_t \beta_t}\left[(1 - \beta_t) \left(1 - \frac{R_i}{p_i}\right) + \frac{(1 - \alpha_t) R_i}{p_i}\left(1 - \frac{\alpha_t \beta_t R_i}{p_i}\right)\right]\\
    &= \frac{\alpha_t}{1 - \alpha_t \beta_t}\left[1 - \beta_t - \frac{R_i}{p_i} + \frac{\beta_t R_i}{p_i} + \frac{R_i}{p_i} - \frac{\alpha_t R_i}{p_i} - \frac{\alpha_t \beta_t R_i^2}{p_i^2} + \frac{\alpha_t^2 \beta_t R_i^2}{p_i^2}\right]\\
    &= \frac{\alpha_t}{1 - \alpha_t \beta_t}\left[\left(1 - \beta_t - \frac{\alpha_t R_i}{p_i}\right) + \frac{\beta_t R_i}{p_i}\left(1 - (1 - \alpha_t)\frac{\alpha_t R_i}{p_i}\right)\right].
\end{align*}

We conclude that $\delta_{t}^{(t+1)} > 0$ since $1 - \beta_t - \frac{\alpha_t R_i}{p_i} > 0$. We also deduce from the form of $x_{t+1}$ that for $l < t$

\begin{equation}
    \delta_l^{(t+1)} = \left(1 - \frac{\alpha_t \beta_t R_i}{p_i}\right)\frac{(1 - \alpha_t)}{1 - \alpha_t \beta_t} \delta_l^{(t)},
\end{equation}

so these coefficients are positive as well. Since they also sum to $1$, it implies that $x_t$ is a convex combination of the $v_l$ for $l < t$, which means that $\psi(x_t) \leq \hat{\psi}_t$ since $\psi$ is convex. Now, we can properly express $\hat{\psi}_{t+1}$:

\begin{align*}
    \esp{\hat{\psi}_{t+1}} &= \esp{\frac{\alpha_t R_i}{p_i}\psi(v_{t+1})} + \esp{\frac{\alpha_t(1 - \beta_t)}{1 - \alpha_t \beta_t}\left(1 - \frac{R_i}{p_i}\right) + \frac{\alpha_t R_i}{p_i} \left(1 - \frac{\alpha_t \beta_t R_i}{p_i}\right)\frac{(1 - \alpha_t)}{1 - \alpha_t \beta_t}} \psi(v_t)\\
    &+ \sum_{l=0}^{t-1} \esp{\left(1 - \frac{\alpha_t \beta_t R_i}{p_i}\right)\frac{(1 - \alpha_t)}{1 - \alpha_t \beta_t}} \delta^{(t)}_l\psi(v_l) \\
    &= \esp{\frac{\alpha_t R_i}{p_i}\psi(v_{t+1})} + \frac{\alpha_t(1 - \beta_t)}{1 - \alpha_t \beta_t}\esp{1 - \frac{R_i}{p_i}} \psi(v_t) +  \frac{(1 - \alpha_t)}{1 - \alpha_t \beta_t} \esp{1 - \frac{\alpha_t \beta_t R_i}{p_i}} \sum_{l=0}^{t} \delta^{(t)}_l \psi(v_l).
\end{align*}

We note $S_R = \esp{\frac{R_i}{p_i}} = \sum_{i=1}^n R_i$ and develop the first term:

\begin{align*}
    \esp{\frac{\alpha_t R_i}{p_i}\psi(v_{t+1})} &= \alpha_t \sum_{i=1}^n \left[R_i\psi_i(v_{t+1}^{(i)}) + R_i\sum_{j \neq i} \psi_j(w_t^{(j)})\right]\\
    &= \alpha_t \sum_{i=1}^n R_i \left[\psi_i(v_{t+1}^{(i)}) - \psi_i(w_{t}^{(i)})\right] + \alpha_t S_R \psi(w_t)\\
    &= \alpha_t \psi(\tilde{v}_{t+1}) + \alpha_t(S_R - 1)\psi(w_t),
\end{align*}

where the last line comes from the fact that either $R_i = 1$ or $\psi_i = 0$. Finally, we use the fact that $w_t = \frac{\beta_t(1 - \alpha_t)}{1 - \alpha_t\beta_t} x_t + \frac{1 - \beta_t}{1 - \alpha_t\beta_t} v_t$, the convexity of $\psi$ and the fact that $\psi(x_t) \leq \hat{\psi}_t$ to obtain:
\begin{align*}
    \esp{\hat{\psi}_{t+1}} &= \alpha_t \psi(\tilde{v}_{t+1}) + \alpha_t(S_R-1)\frac{(1 - \beta_t)}{1 - \alpha_t\beta_t}\psi(v_t) + \alpha_t(S_R-1) \frac{\beta_t(1 - \alpha_t)}{1 - \alpha_t\beta_t}\psi(x_t) \\
    &+ \frac{\alpha_t(1 - \beta_t)}{1 - \alpha_t \beta_t}\left(1 - S_R\right) \psi(v_t) +  \frac{(1 - \alpha_t)}{1 - \alpha_t \beta_t} \left(1 -  \alpha_t \beta_t S_R\right) \hat{\psi}_t\\
    &\leq \alpha_t \psi(\tilde{v}_{t+1}) + \frac{1 - \alpha_t}{1 - \alpha_t \beta_t}\left[(S_R-1)\alpha_t \beta_t + (1 - \alpha_t \beta_t S_R)\right] \hat{\psi}_t\\
    &\leq \alpha_t \psi(\tilde{v}_{t+1}) + (1 - \alpha_t)\hat{\psi}_t.
\end{align*}

This finishes the proof of the lemma.
\end{proof}

\section{Algorithm derivation}
\label{app:algo_derivations}
\subsection{From edge variables to node variables}
Taking the dual formulation implies that variables are associated with edges rather than nodes. Although it could be possible to work with edge variables, it is generally inefficient. Indeed, the algorithm needs variable $Ay_t$ instead of variable $y_t$ for the gradient computation so standard methods work directly with $Ay_t$ \citep{scaman2017optimal, hendrikx2018accelerated}.

The new update equations can be retrieved by multiplying each line of Algorithm~\ref{algo:generalized_apcg} by $A$ on the left. Yet, there is still a $z_{t+1}$ term because of the presence of the proximal update, that is written, for edge $(i,j)$:

\begin{equation}
    v_{t+1}^{(i,j)} = {\rm prox}_{\eta_{i,j} \psi_{i,j}}\left(e_{ij}^T z_{t+1}\right).
\end{equation}

Fortunately, this update only modifies $z_{t+1}$ when $\psi_{i,j} \neq 0$. This means that $z_{t+1}$ is only modified for local computation edges. Since local computation nodes only have one neighbour, the form of $A$ ensures that:

\begin{equation}
    u_{ij}^T A = - \mu_{ij} e_{ij}^T,
\end{equation}

where $u_{ij}$ is the unit vector of size $n(1 + m)$ representing the virtual node j attached to node $i$. In particular, the proximal update can be rewritten:

\begin{align*}
    \left(A v_{t+1}\right)^{(i,j)} &= - \mu_{ij} {\rm prox}_{\eta_{i,j} \psi_{i,j}}\left(- \frac{1}{\mu_{ij}} u_{i,j}^T A z_{t+1}\right)\\
    &= - \mu_{ij} \arg\min_v \frac{1}{2\eta_{i,j}}\|v - \left(- \frac{1}{\mu_{ij}} u_{i,j}^T A z_{t+1}\right)\|^2 + \psi_{i,j}(v)\\
    &= - \mu_{ij} \arg\min_v \frac{1}{2\eta_{i,j}\mu_{ij}^2}\|- \mu_{ij}v - u_{i,j}^T A z_{t+1}\|^2 + f_{i,j}^*(- \mu_{ij} v) - \frac{\mu_{ij}^2}{2L_{j}}\|v\|^2\\
    &= \arg\min_{\tilde{v}} \frac{1}{2\eta_{i,j}\mu_{ij}^2}\|\tilde{v} - u_{i,j}^T A z_{t+1}\|^2 + f_{i,j}^*(\tilde{v}) - \frac{1}{2L_{j}}\|\tilde{v}\|^2\\
    &= {\rm prox}_{\eta_{i,j}\mu_{ij}^2\tilde{f}^*_{i,j}}\left(\left(A z_{t+1}\right)^{(i,j)}\right),
\end{align*}

where $\tilde{f}^*_{i,j}: x \rightarrow f_{i,j}^*(x) - \frac{1}{2L_{j}}\|x\|^2$.

Finally, this can be written:
\begin{equation}
    {\rm prox}_{\eta_{i,j}\mu_{ij}^2\tilde{f}^*_{i,j}}\left(\left(A z_{t+1}\right)^{(i,j)}\right) = {\rm prox}_{L_{j}(\tilde{\nu_{i,j}}^{-1} - 1)^{-1} f^*_{i,j}}\left(\left(1 - \tilde{\nu_{i,j}}\right)^{-1} \left(A z_{t+1}\right)^{(i,j)}\right),
\end{equation}

with $\tilde{\nu_{i,j}} = \tilde{\eta}_{i,j} L_{j}^{-1} = \frac{\rho}{2p_{ij}}\frac{1 + \kappa}{1 + \kappa_i} \leq 1$. Then, we can use the moreau decomposition to formulate the update as a proximal update on the primal function. For the center node, the update can be written:
\begin{align*}
    \left(A v_{t+1}\right)^{(i)} &= \left(A z_{t+1}\right)^{(i)} - \mu_{ij} e_{i,j}^T z_{t+1} + \mu_{ij} {\rm prox}_{\eta_{i,j} \psi_{i,j}}\left(- \frac{1}{\mu_{ij}} u_{i,j}^T A z_{t+1}\right)\\
    &= \left(A z_{t+1}\right)^{(i)} +  \left(A z_{t+1}\right)^{(i,j)} - {\rm prox}_{\eta_{i,j}\mu_{ij}^2\tilde{f}^*_{i,j}}\left(\left(A z_{t+1}\right)^{(i,j)}\right).
\end{align*}

\subsection{Primal proximal updates}
\label{app:algo_derivations_primal_prox}
 Moreau identity~\citep{parikh2014proximal} provides a way to retrieve the proximal operator of $f^*$ using the proximal operator of $f$, but this does not directly apply to $\tilde{f}^*_{i,j}$, making its proximal update hard to compute when no analytical formula is available to compute $\tilde{f}^*_{i,j}$. Fortunately, the proximal operator of $\tilde{f}^*_{i,j}$ can be retrieved from the proximal operator of $f^*_{i,j}$. More specifically, if we note $\tilde{\eta_{i,j}} = \eta_{i,j}\mu_{ij}^2$ then we can also express the update only in terms of $f_{i,j}^*$:

\begin{align*}
    {\rm prox}_{\tilde{\eta}_{i,j} \tilde{f}^*_{i,j}}\left(\left(A z_{t+1}\right)^{(i,j)}\right) &= \arg\min_v \frac{1}{2\tilde{\eta}_{i,j}} \|v - \left(A z_{t+1}\right)^{(i,j)}\|^2 + \tilde{f}^*_{i,j}(v) - \frac{1}{2L_j}\|v\|^2\\
    &= \arg\min_v \frac{1}{2}\left(\tilde{\eta}_{i,j}^{-1} - L_j^{-1}\right)\|v\|^2 - \tilde{\eta}_{i,j}^{-1} v^T \left(A z_{t+1}\right)^{(i,j)} + \tilde{f}^*_{i,j}(v) \\
    &= \arg\min_v \frac{1}{2\left(\tilde{\eta}_{i,j}^{-1} - L_j^{-1}\right)^{-1}}\|v - \left(1 - \tilde{\eta}_{i,j} L_{j}^{-1}\right)^{-1} \left(A z_{t+1}\right)^{(i,j)}\|^2 + \tilde{f}^*_{i,j}(v) \\
    &= {\rm prox}_{\left(\tilde{\eta}_{i,j}^{-1} - L_{j}^{-1}\right)^{-1} f^*_{i,j}}\left(\left(1 - \tilde{\eta}_{i,j} L_{j}^{-1}\right)^{-1} \left(A z_{t+1}\right)^{(i,j)}\right).
\end{align*}
    
Then, we use the identity:
\begin{equation}
    {\rm prox}_{\left(\eta f\right)^*}(x) = \eta {\rm prox}_{\eta^{-1} f^*}\left(\eta^{-1} x\right)
\end{equation}

and the Moreau identity to write that:
\begin{equation}
    {\rm prox}_{\eta f^*}(x) = x -  \eta {\rm prox}_{\eta^{-1} f}\left(\eta^{-1} x\right).
\end{equation}

This allows us to retrieve the proximal operator on $\tilde{f}^*_{i,j}$ using only the proximal operator on $f_{i,j}$:
\begin{equation}
    \left(1 - \tilde{\eta}_{i,j} L_{j}^{-1}\right) {\rm prox}_{\tilde{\eta}_{i,j} \tilde{f}^*_{i,j}}\left(\left(A z_{t+1}\right)^{(i,j)}\right) = \left(A z_{t+1}\right)^{(i,j)} - \tilde{\eta}_{i,j} {\rm prox}_{\left(\tilde{\eta}_{i,j}^{-1} - L_{j}^{-1}\right) f} \left(\tilde{\eta}_{i,j}^{-1} \left(A z_{t+1}\right)^{(i,j)}\right).
\end{equation}

\subsection{Projection of virtual edges}
Theorem~\ref{app:generalized_apcg} requires that for any coordinate $i$, either the proximal part $\psi_i = 0$ or the coordinate is not affected by the change of norm. In our case, $\psi_{i,j} = 0$ as long as $(i,j)$ is a communication edge. Lemma~\ref{lemma:acrossa} is a small result that will prove to be very useful as it shows that the projection condition is satisfied by virtual edges. 

\begin{lemma}
\label{lemma:acrossa}
If $(i,j)$ is a virtual edge, then $A^+Ae_{i,j} = e_{i,j}$.
\end{lemma}

\begin{proof}
Let $x \in \mathbb{R}^{|E|}$ such that $Ax = 0$. From the definition of $A$, either $x = 0$ or the support of $x$ is a cycle of the graph. Indeed, for any edge $(i,j)$, $Ae_{i,j}$ has non-zero weights only on node $i$ and $j$. Virtual nodes have degree one, so virtual edges are parts of no cycles and therefore $e_{i,j}^T x = 0$ for all virtual edges $(i,k)$. Operator $A^+A$ is the projection operator on the orthogonal the kernel of $A$, so it is the identity on virtual edges. 
\end{proof}

\subsection{Smooth case}
The last step to have a complete algorithm is to choose initial values for parameters $A$ and $B$, that will control their increase. In the smooth case (when the dual is strongly convex), we can simply choose $\alpha_t = \beta_t = \rho = \frac{\sqrt{\sigma}}{S}$ with the notations of Section~\ref{app:generalized_apcg}. This yields $\frac{a_{t+1}}{B_{t+1}} = \frac{\rho}{\sigma}$, allowing to set all the parameters of the algorithm. Following the same calculations as in~\citet{hendrikx2018accelerated}, we obtain $A_t = \left(1 - \rho\right)^{-t}$ and $B_t = \sigma A_t$.

The value of $S$ is obtained by remarking that $q_A$ is $\left[\mu_{ij}^2\left(\Sigma_i^{-1} + \Sigma_j^{-1}\right)\right]^{-1}$ smooth in the direction $(i,j)$. Note that $q_A$ is also $\lambda_{\min}^+\left(A^T\Sigma^{-1}A\right)$ strongly convex in the norm $A^+A$.

Then, Theorem~\ref{thm:gen_apcg} and Lemma~\ref{lemma:acrossa} yield:
    \begin{equation*}
    B_t\esp{\|\tilde{v}_t - \theta^*_A\|^2_{A^+A}} + 2 A_t \left[\esp{F^*(\tilde{x}_t)} - F^*(\theta^*_A)\right] \leq C_0,
    \end{equation*}
with $v_t = A\tilde{v}_t$ and $x_t = A\tilde{v}_t$. Then, we use the fact that for any $x$, $F^*(x) = F^*(A^+Ax)$ to write that $\esp{F^*(\tilde{x}_t)} = \esp{F^*(A^+x_t)}$. Following \citet{lin2015accelerated}, the primal optimal point $\theta^*$ can be retrieved as $\theta^* = \nabla q_A(\theta^*_A) = \Sigma^{-1}A\theta^*_A$, where $\theta^*_A$ is the optimal dual parameter. Finally,$$\lambda_{\max}(A^T\Sigma^{-2}A)^{-1} \|\theta_t - \theta^*\|^2 \leq \lambda_{\max}(A^T\Sigma^{-2}A)^{-1} \|\Sigma^{-1}v_t - \Sigma^{-1}A\theta^*_A\|^2 \leq \|\tilde{v}_t - \theta^*_A\|^2_{A^+A},$$ which finishes the proof of Theorem~\ref{thm:rate_adfs}.

\subsection{Non-smooth case}
\label{app:non_smooth_adfs}

\begin{algorithm}
\caption{NS-ADFS}
\label{algo:ns_adfs}
\begin{algorithmic}
\STATE $x_0 = 0$, $v_0 = 0$, $t = 0$, $A_0 = 0$, $\eta_{ij} = \frac{\mu_{ij}^2}{p_{ij}}$
\WHILE{$t < T$}
\STATE $A_{t+1} = A_t + \frac{1}{2S^2}\left(1 + \sqrt{1 + 4S^2A_t}\right)$
\STATE $a_{t+1} = A_{t+1} - A_t$, $\alpha_t = \frac{a_{t+1}}{A_{t+1}}$
\STATE $y_t = (1 - \alpha_t)x_t + \alpha_t v_t$
\STATE Sample $(i,j)$ with probability $p_{i,j}$
\STATE $v_{t+1} = z_{t+1} = v_t - a_{t+1} \eta_{ij} W_{ij}\Sigma^{-1}y_t$
\IF{$(i,j)$ is a computation edge}
\STATE $v_{t+1}^{(j)} = {\rm prox}_{a_{t+1} \eta_{ij} f^*_{i,j}}\left(z_{t+1}^{(i,j)}\right)$
\STATE $v_{t+1}^{(i)} = z_{t+1}^{(i)} + z_{t+1}^{(i,j)} - v_{t+1}^{(i,j)}$
\ENDIF
\STATE $x_{t+1} = y_t + \frac{\rho R_{ij}}{p_{ij}}(v_{t+1} - v_t)$
\ENDWHILE
\STATE \textbf{return} $\theta_t = \Sigma^{-1}v_t$
\end{algorithmic}
\end{algorithm}

For non-smooth function, the Fenchel conjugate is not strongly convex so $\sigma_A = 0$ and so $\beta_t = 0$. Then, we can choose $B_t = 1$ for all $t$ and $A_t$ such that $A_0 = 0$ and an easy recursion yields: $$A_{t+1} = A_t + \frac{1}{2S^2}\left(1 + \sqrt{1 + 4S^2A_t}\right)$$.

The other coefficients can be computed from the fact that $B_t$ is constant and equal to $1$ and that $a_{t+1} = A_{t+1} - A_t = \frac{1}{2S^2}\left(1 + \sqrt{1 + 4S^2A_t}\right)$, also yielding a formula for $\alpha_t$. In particular, $A_t \geq \frac{t^2}{4S^2}$, which leads to the following theorem:

\begin{theorem}
For non-smooth functions, sequence $x_t$ as generated by Algorithm~\ref{algo:ns_adfs} guarantees:
\begin{equation}
 \mathbb{E}[F^*(x_t)] - F^*(A\theta^*) \leq \frac{1}{t^2} \frac{2S^2}{\lambda_{\min}^+(A^T A)}\|A\theta^*\|^2_{A^+A}.
\end{equation}
\end{theorem}

Note that $\alpha_t = \grando{t^{-1}}$, and $\frac{a_{t+1}}{B_{t+1}} = \grando{t}$. We have a $\sqrt{\varepsilon^{-1}}$ convergence rate. The constant is $\frac{\lambda_{\min}^+\left(A^TA\right)}{S}$, which is very related to the constant for the smooth case, simply that the $\Sigma^{-1}$ factor is removed. Therefore, we directly deduce that if we choose $\mu_{ij} = \frac{\lambda_{\min}^+(L)}{1 + m}$ when $(i,j) \in E^{\rm comp}$ then we get: $$ \lambda_{\min}^+(A^TA) \geq \frac{\lambda_{\min}^+(L)}{2(m + 1)}.$$

Optimizing parameter $\rho$ in order to minimize time yields $\rho_{\rm comp} = \rho_{\rm comm}$ again, now leading in the homogeneous case to $$p_{\rm comm} = \left(1 + \sqrt{m\tilde{\gamma}} \right)^{-1},$$ and so the time taken by the algorithm to reach error $\varepsilon$ is $$ \grando{\left[(1 + \tau)\tilde{\gamma}^{-1/2} + \sqrt{m}\right]\varepsilon^{-1/2}},$$ where the $1 + \tau$ factor comes again from the limitation $p_{\rm comp} > p_{\rm comm}$.

Note that all parameters have been multiplied by $A$ on the left, just as for Algorithm~\ref{algo:sc_adfs}. Since the scale of $A$ is a parameter of the algorithm, it is important to count the starting error with $A\theta^*$. This transformation did not affect the rate in the smooth case because we had exponential convergence but it is necessary here to obtain consistent rates. Finally, we obtain rates of convergence only for the dual function suboptimality. Therefore, we cannot apply the same trick we applied in the smooth case to derive the error in the primal parameters because we do not have a bound on $\|A v_t - A \theta^*\|^2$.

\section{Detailed average time per iteration proof}
\label{appendix:average_time}
The goal of this section is to prove Theorem~\ref{thm:synchronization_cost}. The proof is an extension of the proof of Theorem~2 from \citet{hendrikx2018accelerated}. Similarly, we note $t$ the number of iterations that the algorithm performs and $\tau_c^{ij}$ the random variable denoting the time taken by a communication on edge $(i,j)$. Similarly, $\tau_l^i$ denotes the time taken by a local computation at node $i$. Then, we introduce the random variable $X^t(i, w)$ such that if edge $(i,j)$ is activated at time $t+1$  (with probability $p_{ij}$), then for all $w \in \mathbb{N}^*$:
\begin{equation*}
    X^{t+1}(i,w) = X^t(i, w - \tau_c^{ij}(t)) + X^t(j, w - \tau_c^{ij}(t)),
\end{equation*}

where $\tau_c^{ij}(t)$ is the realization of $\tau_c^{ij}$ corresponding to the time taken by activating edge $(i,j)$ at time $t$. If node $i$ is chosen for a local computation, which happens with probability $p^{\rm comp}_i$ then $X^{t+1}(i, w + \tau_l^i(t)) = X^t(i, w)$ for all $w$. Otherwise, $X^{t+1}(j, w) = X^t(j, w)$ for all $w$. At time $t=0$, $X^0(i, 0) = 1$ and $X^0(i,w) = 0$ for all $w$. Lemma~\ref{lemma:ptgttheta} gives a bound on the probability that the time taken by the algorithm to complete $t$ iterations is greater than a given value, depending on variables $X^t$. Note that a Lemma similar to the one in \citet{hendrikx2018accelerated} holds although variable $X$ has been modified. 

\begin{lemma}
\label{lemma:ptgttheta}
We note $T_{\max}(t)$ the time at which the last node of the system finishes iteration $t$. Then for all $\nu > 0$:

\begin{equation*}
    \mathbb{P}\left(T_{\max}(t) \geq \nu t\right) \leq \sum_{w \geq \nu t} \sum_{i=1}^n \mathds{E}\left[X^t(i, w)\right].
\end{equation*}

\end{lemma}

\begin{proof}
We first prove by induction on $t$ that for any $i \in \{1, .., n\}$: 

\begin{equation}
\label{eq:rec_ti}
    T_i(t) = \max_{w \in \mathbb{N}, X^t(i,w) > 0} w.
\end{equation}

To ease notations, we write $w_{\max}(i,t) = \max_{w \in \mathbb{N}, X^t(i,w) > 0} w$. The property is true for $t=0$ because $T_i(0) = 0$ for all $i$.

We now assume that it is true for some fixed $t > 0$ and we assume that edge $(k,l)$ has been activated at time $t$. For all $i \notin \{k, l\}$, $T_i(t+1) = T_i(t)$ and for all $w \in \mathbb{N}^*$, $X^{t+1}(i, w) = X^t(i, w)$ so the property is true. Besides, if $j\neq l$

\begin{align*}
    w_{\max}(k, t+1) &= \max_{w \in \mathbb{N^*}, X^t(k,w - \tau_c(t)) + X^t(l,w - \tau_c^{kl}(t)) > 0} w \\
    &= \max_{w \in \mathbb{N}, X^t(i,w) + X^t(i,w) > 0} w + \tau_c^{kl}(t) \\
    &= \tau_c(t) + \max\left(w_{\max}(k, t), w_{\max}(l, t)\right)\\
    &= \tau_c^{kl}(t) + \max \left(T_k(t), T_l( t)\right) = T_k(t+1).
\end{align*}

Similarly if $k=l$ (a local computation is performed at iteration $t$), then $w_{\max}(k, t+1) = \tau_l^k(t) + w_{\max}(k, t) = T_k(t) + \tau_l^k(t) = T_k(t+1)$. Then, we use the union bound and the the fact that having $X^t(i, w) > 0$ is equivalent to having $X^t(i, w) \geq 1$ since $X^t(i, w)$ is integer valued to show that:
\begin{align*}
    \mathbb{P}\left(T_{\max}(t) \geq \nu t\right)& = \mathbb{P}\left(\max_{w, \sum_{i=1}^n X_i^t(w) > 0 } w\geq \nu t\right) \leq \mathbb{P}\left(\cup_{w \geq \nu t } \sum_{i=1}^n X_i^t(w) \geq 1 \right)\leq \sum_{w \geq \nu t} \mathbb{P}\left(\sum_{i=1}^n X_i^t(w) \geq 1\right),
\end{align*}
so using the Markov inequality yields: 
\begin{equation}
    \mathbb{P}\left(T_{\max}(t) \geq \nu t\right) \leq \sum_{w \geq \nu t} \sum_{i=1}^n \mathbb{E}\left[X_i^t(w)\right].
\end{equation}

\end{proof}

Variables $X_i^t$ are obtained by linear recursions, so Lemma~\ref{lemma:ptgttheta} allows us to bound the growth of variables with a simple recursion formula instead of evaluating a maximum. We write $p_i^{\rm comp}$ and $p_i^{\rm comm}$ the probability that node $i$ performs a computation (respectively communication) update at a given time step, and $p_i = p_i^{\rm comp} + p_i^{\rm comm}$. We introduce $\underbar{p}_{\rm comp} = \min_i p_i^{\rm comp}$ and $\bar{p}_{\rm comp} = \max_i p_i^{\rm comp}$ (and the same for communication probabilities).

\begin{lemma}
\label{lemma:sum_binom}
For all $i$, and all $\nu > 0$, if $\frac{1}{2} \geq \underbar{p}_{\rm comp} = \bar{p}_{\rm comp} \geq \bar{p}_{\rm comm}$ then:

\begin{equation}
    \sum_{w \geq \left(\nu_c + \nu_l\right) t} \sum_{i=1}^n \mathbb{E}\left[X^t(i, w)\right] \rightarrow 0 \text{ when } t \rightarrow \infty
\end{equation}

with $\nu_c = 6 p_c \tau_c$ and $\nu_l = 9 p_l \tau_l$ where $p_c = 4 \bar{p}_{\rm comm}$ and $p_l = \bar{p}_{\rm comp}$.
\end{lemma}

Note that the constants in front of the $\nu$ parameters are very loose.

\begin{proof}

Taking the expectation over the edges that can be activated gives, with $\tau_c^{ij}(\tau)$ the probability that $\tau_c^{ij}$ takes value $\tau$ (and the same for $\tau_l$):

\begin{align*}
    \mathbb{E}\left[X^{t+1}(i,w)\right] = \left(1 - p_i\right) \mathbb{E}\left[X^{t}(i,w)\right] + &p_{\rm comm}\sum_{j=1}^n p_{ij} \sum_{\tau = 0}^\infty \tau_c^{ij}(\tau) \left(\mathbb{E}\left[X^{t}(i, w - \tau)\right] + \mathbb{E}\left[X^{t}(j, w - \tau)\right]\right) \\
    & + p_i^{\rm comp} \sum_{\tau = 0}^\infty \tau_l^{ij}(\tau) \mathbb{E}\left[X^{t}(i, w - \tau)\right].
\end{align*}

In particular, for all $i$, $\mathbb{E}\left[X^{t+1}(i,w)\right] \leq \bar{X}^t(w)$ where $\bar{X}^0(w) = 1$ if $w = 0$ and:

\begin{equation}
    \bar{X}^{t+1}(w) = \left(1 - \underbar{p} \right)\bar{X}^{t}(w) + 2 \bar{p}_{\rm comm} \sum_{\tau = 0}^\infty \tau_c^{\max}(\tau) \bar{X}^{t}(w - \tau) + \bar{p}_{\rm comp} \sum_{\tau = 0}^\infty \tau_l^{\max} (\tau) \bar{X}^{t}(w - \tau).
\end{equation}

with $\tau_c^{\max}(\tau) = \max_{ij} \tau_c^{ij}(\tau)$ (and the same for $\tau_l$). We now introduce $\phi^t(z) = \sum_{w \in \mathbb{N}} z^w \bar{X}^t(w)$. We note $\phi_c$ and $\phi_l$ the generating functions of $\tau_c^{\max}(\tau)$ and $\tau_l^{\max}(\tau)$. A direct recursion leads to:
\begin{equation}
    \phi^t(z) = \left(1 - \underbar{p}_{\rm comm} - \underbar{p}_{\rm comp} +  \bar{p}_{\rm comp} \phi_l(z) + 2 \bar{p}_{\rm comm} \phi_c(z) \right)^t = \left(\phi^1(z)\right)^t.
\end{equation}

We note $\phi_{bin}(p,t)$ the generating function associated with the binomial law of parameters $p$ and $t$. With this definition, we have:

\begin{equation}
    \phi_{bin}(p_c,t)(\phi_c(z)) \phi_{bin}(p_l,t)(\phi_l(z)) = \left[ (1 - p_c)(1 - p_l) + (1 - p_c)p_l \phi_l(z) + (1 - p_l)p_c \phi_c(z) + p_c p_l \phi_c(z) \phi_l(z)\right]^t,
\end{equation}
so we can define:

\begin{equation}
     \phi^t_+(z) = (1 + \delta)^t \phi_{bin}(p_c,t)(\phi_c(z)) \phi_{bin}(p_l,t)(\phi_l(z)),
\end{equation}

where $p_c$, $p_l$ and $\delta$ are such that:

\begin{equation*}
    \frac{p_c}{1 - p_c} \geq 2 \frac{\bar{p}_{\rm comm}}{1 - \underbar{p}}, \ \    \frac{p_l}{1 - p_l} = \frac{\bar{p}_{\rm comp}}{1 - \underbar{p}} \text{ and } \delta \geq \frac{1 - \underbar{p}}{(1 - p_c)(1 - p_l)} - 1. 
\end{equation*}

Since $\bar{p}_{\rm comp} = \underbar{p}_{\rm comp}$ then $\underbar{p} \geq \bar{p}_{\rm comp}$. Therefore, these conditions are satisfied for $p_c$ and $p_l$ as given by Lemma~\ref{lemma:sum_binom} and $\delta = (1 - p_c)^{-1} - 1$. Then $(1 + \delta)(1 - p_c)(1 - p_l) \geq 1 -  \underbar{p}$, $(1 + \delta)(1 - p_c)p_l \geq \bar{p}_{\rm comp}$ and  $(1 + \delta)(1 - p_l)p_c \geq 2\bar{p}_{\rm comm}$. This means that if we write $\phi^1(z) = a_0 + a_c \phi_c(z) + a_l \phi_l(z)$ and $\phi^1_+(z) = b_0 + b_c \phi_c(z) + b_l \phi_l(z)$ then $b_0 \geq a_0$, $b_c \geq a_c$ and $b_l \geq a_l$. In particular, all the coefficients of $\phi^t$ are smaller than the coefficients of $\phi^t_+$ where both functions are integral series. Therefore, if we call $Z_t$ the random variables associated with the generating function $(1 + \delta)^{-t}\phi^t_+$ then for all $i, t, w$:

\begin{equation}
\label{eq:control_EXt}
\mathbb{E}\left[X^t(i, w)\right] \leq (1 + \delta)^t \mathbb{P}\left(Z_t = w\right)    
\end{equation}

where $Z_t = Z_c^t + Z_l^t =  Bin(p_c, t)(Z_c) + Bin(p_l, Z_l)(\tau_l)$ where $Z_c$ and $Z_l$ are the random variables modeling the time of one communication or computation update. We can then use the bound $p(Z_t \geq (\nu_c + \nu_l) t) \leq p(Z_c^t \geq \nu_c t) + p(Z_l^t \geq \nu_l t)$. This way, we can bound the \emph{communication} and \emph{computation} costs independently. Then, we write a Chernoff bound, i.e. for any $\lambda > 0$:

\begin{align*}
    \mathbb{P}\left(Z_c^t \geq \nu t \right) \leq e^{-\lambda \nu t} \esp{e^{\lambda Z_c^t}} = e^{-\lambda \nu t} \esp{e^{\lambda Z_c}}^t =e^{-\lambda \nu t} \left[1 - p_c + p_c \sum_{\tau = 0}^\infty p_c(\tau)e^{\lambda \tau} \right]^t,
\end{align*}

where $S_c$ is the sum of $t$ i.i.d. random variables drawn from $\tau_c$. If $Z_c = \tau_c$ with probability $1$ (deterministic delays) then this reduces to:
\begin{align*}
    \mathbb{P}\left(Z_c^t \geq \nu_c t \right) \leq e^{-\lambda \nu_c t} \left[1 - p_c + p_c e^{\lambda \tau_c} \right]^.
\end{align*}

Finally, we take $\nu_c = k p_c \tau_c$, $\lambda = \frac{1}{\tau_c}\ln(k)$ and we use the basic inequality $\ln(1 + x) \geq \frac{x}{1 + x}$ to show that:

\begin{equation}
    - \ln\left[\mathbb{P}\left(Z_c^t \geq \nu_c t \right)\right] \geq t\left[\lambda \nu_c - p_c \left(e^{\lambda \tau_c} - 1\right) \right] \geq t (k(\ln(k) - 1) - 1)p_c.
\end{equation}

Using the same log inequality and the fact that $p_c \geq \frac{1}{2}$ yields:
\begin{equation}
    \ln\left(1 + \delta\right) = - \ln(1 - p_c) \leq \frac{p_c}{1 - p_c} \leq 2 p_c
\end{equation}

Therefore, choosing $k = 6$ ensures that  $k(\ln(k) - 1) - 1 \geq 3$ and so: 
\begin{equation}
    (1 + \delta)^t \mathbb{P}\left(Z_c^t \geq \nu_c t \right) \leq e^{-t p_c}
\end{equation}. 

We can apply the same reasoning to $Z_l^t$, and the bound is still valid with $k = 9$ because $p_l = \bar{p}_{\rm comp} \geq \bar{p}_{\rm comm} = p_c / 4$. We finish the proof by using Equation~\ref{eq:control_EXt}.
\end{proof}

\section{Algorithm performances}
\label{app:algo_perfs}
\adfs~has a linear convergence rate because it results from using generalized APCG. Yet, it is not straightforward to derive an efficient set of hyperparameters that lead to a rate that can be easily interpreted. The goal of this section is to choose such parameters. 

\subsection{Smallest eigenvalue of the Laplacian of the augmented graph}
The strong convexity of $q_A$ in the norm $A^+A$ is equal to $\lambda_{\min}^+\left(A^T\Sigma^{-1}A\right)$, the smallest non-zero eigenvalue of $A^T\Sigma^{-1}A$. The goal of this section is to prove that for a meaningful choice of $\mu$, the smallest eigenvalue of the Laplacian of the augmented graph is not too small compared to the Laplacian of the actual graph. More specifically, we prove the following result:

\begin{lemma}
If $\mu_{ij}$ are such that $\mu_{ij}^2 = \frac{\lambda_{\min}^+(L)}{\sigma( 1 + \kappa_i)} L_j$ and $\sigma$, $\kappa$ are such that for all $i$, $\sigma \geq \sigma_i$ and $\kappa \geq \kappa_i$ then:

\begin{equation}
    \lambda_{\min}^+(\tilde{L}) \geq \frac{\lambda_{\min}^+(L)}{2\sigma(1 + \kappa)}
\end{equation}
\end{lemma}

\begin{proof}
All non-zero singular values of a matrix $M^TM$ are also singular values of the matrix $MM^T$, and so $\lambda_{\min}^+(A^T\Sigma^{-1} A) = \lambda_{\min}^+\left(\Sigma^{-1/2} A A^T \Sigma^{-1/2}\right)$. We note $\tilde{L} = \Sigma^{-1/2} A A^T \Sigma^{-1/2}$.

Then, we note $\mu_{ij}^2$ the weight of the \emph{computing edge} $(i,j)$ and $M$ the diagonal matrix of size $m$ which is such that $M_{i,j} = \mu_{ij}^2$. $M_{n,m}$ is the matrix of size $n \times nm$ such that $e_i^T M_{n,m} e_{i,j} = \mu_{ij}^2$ and all other entries are equal to 0. Finally, $\tilde{S}$ is the diagonal matrix of size $n$ such that $\tilde{S}_i = \sum_{j \in V_i} \mu_{ij}^2$. All \emph{communication nodes} are linked by the true graph, whereas all \emph{computing nodes} are linked to their corresponding communication node. Then, if we note $L$ the Laplacian matrix of the original true graph, the rescaled Laplacian matrix of the augmented graph $G^+$ writes:

\begin{equation}
\tilde{L} = \Sigma^{-1/2}\begin{pmatrix} L + \tilde{S} & - M_{n,m} \\ - M_{n,m}^T & M \end{pmatrix}\Sigma^{-1/2}
\end{equation}

Therefore, if we split $\Sigma$ into two diagonal blocks $D_1$ (for the communication nodes) and $D_2$ (for the computation nodes) and apply the block determinant formula, we obtain: 

\begin{align*}
    &det(D_1^{-\frac{1}{2}}A^TA D_1^{-\frac{1}{2}} - \lambda Id) = det(D_2^{-1}M - \lambda Id)\\
    &\times det(D_1^{-1/2} L D_1^{-1/2} + D_1^{-1} \tilde{S} - \lambda Id -\\
    &D_1^{-\frac{1}{2}} M_{n,m} D_2^{-\frac{1}{2}} \left(D_2^{-1} M - \lambda Id\right)^{-1} D_2^{-\frac{1}{2}} M_{n,m}^T D_1^{-\frac{1}{2}}) 
\end{align*}

Then, we choose $M$ such that $D_2^{-1} M = diag(\alpha_1, ..., \alpha_n)$, meaning that $\mu_{ij}^2 = \alpha_i L_{j}$. With this choice, $D_1^{-\frac{1}{2}} M_{n,m} D_2^{-\frac{1}{2}} \left(D_2^{-1} M - \lambda Id\right)^{-1} D_2^{-\frac{1}{2}} M_{n,m}^T D_1^{-\frac{1}{2}}$ is a diagonal matrix where the $i$-th coefficient is equal to

\begin{equation}
    \frac{1}{\sigma_i}\sum_{j \in V_i} \mu_{ij}^4 \frac{1}{\mu_{ij}^2 - L_{j} \lambda} = \kappa_i \frac{\alpha^2_i}{\alpha_i - \lambda},
\end{equation}

where $\kappa_i = \frac{S_i}{\sigma_i}$ and $S_i = \sum_{j \in V_i} L_{j}$. On the other hand, $D_1^{-1}S$ is also a diagonal matrix where the $i$-th entry is equal to $\alpha \kappa_i$. Therefore, the solutions of $det(\tilde{L} - \lambda Id) = 0$ are $\lambda = \alpha_i$ and the solutions of:
\begin{equation}
    \label{eq:det_deltalambda}
    det(D_1^{-1/2} L D_1^{-1/2} - \Delta_\lambda) = 0
\end{equation}

with $\Delta_\lambda$ a diagonal matrix such that $(\Delta_\lambda)_{i,i} = \left(\frac{1}{\sigma_i}\sum_{j=1}^{m} \mu_{ij}^4\left( \frac{1}{\mu_{ij}^2 - L_{j} \lambda} - 1  \right) + \lambda \right)$. All the entries of $\Delta_\lambda$ grow with $\lambda$, meaning that the smallest solution $\lambda^*$ of Equation~\eqref{eq:det_deltalambda} is lower bounded by the smallest solution of:

\begin{equation}
    \det\left(\frac{\lambda_{\min}^+(L)}{\sigma}Id - \Delta_\lambda\right)
\end{equation}

If $\alpha_i > 0$ and we choose $\lambda \neq \alpha_i$, then the other singular values of $\tilde{L}$ are lower bounded by the minimum over all $i$ of the solution of:

\begin{equation}
    \nu - \left(\frac{1}{\sigma_i}\sum_{j=1}^{m} \mu_{ij}^4\left( \frac{1}{\mu_{ij}^2 - L_{j} \lambda} - 1  \right) + \lambda \right) = 0
\end{equation}

where $\nu = \frac{\lambda_{\min}^+(L)}{\sigma}$ which, with our choice of $\mu_{ij}$ gives:

\begin{equation}
    \nu - \left(\alpha_i \kappa_i\left(\frac{\alpha_i}{\alpha_i - \lambda } - 1\right) + \lambda \right) = 0
\end{equation}

that can be rewritten:
\begin{equation}
    \lambda^2 - \lambda\left(\nu + \alpha_i(\kappa_i + 1)\right) + \alpha_i \nu  = 0
\end{equation}

Therefore, noting $\lambda_i^*$ the smallest solution of this system for a given $i$:

\begin{equation}
    \lambda_i^* \geq \frac{1}{2}\left(\alpha_i(\kappa_i + 1) + \nu - \sqrt{\left(\nu + \alpha_i(\kappa_i + 1)\right)^2 - 4\nu \alpha_i}\right)
\end{equation}

In particular, we choose $\alpha_i = \frac{\nu}{\kappa_i + 1}$ and use that $\sqrt{1 - x} \leq 1 - \frac{x}{2}$ to show:
\begin{equation}
    \lambda_i^*  \geq \nu \left(1 - \sqrt{1 - \frac{1}{ 1 + \kappa_i}}\right)
    \geq \frac{\nu}{2(1 + \kappa_i)}
\end{equation}

The other eigenvalues are given by the values that zero out diagonal terms of the lower right corner. These are the solutions of $\mu_{ij}^2 = L_{j} \lambda$, yielding $\lambda = \alpha_i \geq \lambda_i^*$. Therefore, $\lambda_{\min}^+(\tilde{L}) \geq \min_i \lambda_i^*$, which finishes the proof.

\end{proof} 

\subsection{Communication rate and local rate}
We know that the rate of the algorithm can be written as the minimum of a given quantity over all edges of the graph. This quantity will be very different whether we consider communication edges or virtual edges. In this section, we give lower bounds for each type of edge, and show that we can trade one for the other by adjusting the probability of communication.

\begin{lemma}
With the choice of parameters of Theorem~\ref{thm:adfs_speed}, parameter $\rho$ satisfies:
\begin{equation}
    \rho \geq \frac{\Delta_p}{\sqrt{2} n}\min\left( p_{\rm comm} \sqrt{\frac{\tilde{\gamma}}{1 + \kappa}}, p_{\rm comp}  \frac{\sqrt{r_\kappa}}{S_{\rm comp}} \right).
\end{equation}
\end{lemma}

\begin{proof}
Recall that the rate $\rho$ is defined as:
\begin{equation}
\label{eq:rate_theta}
\rho^2 = \min_{ij} \frac{p_{ij}^2}{\mu_{ij}^2 e_{ij}^T A^+A e_{ij}} \frac{\lambda_{\min}^+ ( \tilde{L} )}{\sigma_i^{-1} + \sigma_j^{-1}}\end{equation}

Therefore, for ``communication edges" the rate writes:

\begin{equation}
\rho^2_{\rm comm} \geq \Delta_p^2 \min_{ij} \left(\frac{1}{\sigma_i} + \frac{1}{\sigma_j}\right)^{-1} \frac{p_{ij}^2}{\mu_{ij}^2e_{ij}^TA^+Ae_{ij}} \frac{\lambda_{\min}^+ \left( L\right)}{2\sigma(1 + \kappa)}
\end{equation}

If we take $\sigma_i = \sigma$ for all $i$, and plug in the fact that $\mu_{ij}^{-2} = 2$ and $p_{ij}^{2} = p_{\rm comm}^2 / |E|^2$ (corresponding to a homogeneous case) then we obtain:

\begin{equation}
\label{eq:lb_comm}
\rho^2_{\rm comm} \geq \Delta_p^2 \frac{\tilde{\gamma}}{1 + \kappa} \frac{p_{\rm comm}^2}{2n^2}.
\end{equation}

For ''computation edges", we can write:
\begin{equation}
\rho_{\rm comp}^2 \geq \min_{i,j} \frac{p_{ij}^2}{2\left(\sigma_i^{-1} + L_{j}^{-1}\right)}\frac{\sigma (1 + \kappa_i) }{\lambda_{\min}^+ \left(L\right)L_{j}} \frac{\lambda_{\min}^+ \left(L\right)}{\sigma(1 + \kappa)},
\end{equation}

because $e_{ij}^T A^+A e_{ij} = 1$ when $(i,j)$ is a "virtual" edge (because it is part of no cycle). Since $S_{\rm comp} = \frac{1}{n}\sum_{i=1}^n \sum_{j=1}^{m} \sqrt{1 + L_{j}\sigma_i^{-1}}$, this can be rewritten:
\begin{equation}
\rho^2_{\rm comp} \geq \Delta_p^2 \frac{r_\kappa}{2} \frac{p_{\rm comp}^2}{n^2S_{\rm comp}^2}. 
\end{equation}
\end{proof} 

\subsection{Execution time}
Now that we have precised the rate of the algorithm (improvement per iteration), we can bound the time needed to reach precision $\varepsilon$ by plugging in the expected time to execute the schedule. This allows to write the proof of Theorem~\ref{thm:adfs_speed}.

\begin{proof}
Using Theorem~\ref{thm:synchronization_cost} on the average time per iteration, we know that as long as $p_{\rm comp} > p_{\rm comm}$, the execution time of the algorithm verifies the following bound for some $C > 0$ with high probability:

\begin{equation}
    \frac{T}{\log\left(\varepsilon^{-1}\right)} \leq \frac{C}{n\rho} \left(p_{\rm comp} +  \tau p^{\max}_{\rm comm}\right)
\end{equation}

If we rewrite this in terms of $\rho_{\rm comm}$ and $\rho_{\rm comp}$, we obtain:

\begin{equation}
     \frac{T}{\log\left(\varepsilon^{-1}\right)} \leq C \max \left(T_1(p_{\rm comm}), T_2(p_{\rm comm})\right)
\end{equation}

with 
\begin{equation}
    T_1(p_{\rm comm}) = \frac{1}{n \rho_{\rm comm}}(p_{\rm comp} + c_\tau\tau p_{\rm comm}) = \frac{\sqrt{2}}{\Delta_p}\left(\tau - 1 + \frac{c_\tau}{p_{\rm comm}}\right)\sqrt{\frac{1 + \kappa}{\tilde{\gamma}}}
\end{equation}

and 
\begin{equation}
    T_2(p_{\rm comm}) = \frac{S_{\rm comp}}{\Delta_p}\sqrt{\frac{2}{r_\kappa}}  \left( \frac{1 + (c_\tau \tau - 1)p_{\rm comm}}{1 - p_{\rm comm}}\right) =\frac{S_{\rm comp}}{\Delta_p}\sqrt{\frac{2}{r_\kappa}} \left(1 + \tau \frac{p_{\rm comm}}{1 - p_{\rm comm}}\right)
\end{equation}

$T_1$ is a continuous decreasing function of $p_{\rm comm}$ with $T_1 \rightarrow \infty$ when $p_{\rm comm} \rightarrow 0$. Similarly, $T_2$ is a continuous increasing function of $p_{\rm comm}$ such that $p_{\rm comm} \rightarrow \infty$ when $p_{\rm comm} \rightarrow 1$. Therefore, the best upper bound on the execution time is given by taking $p_{\rm comm} = p^*$ where $p^*$ is such that $T_1(p^*) = T_2(p^*)$ and so $\rho_{\rm comm}(p^*) = \rho_{\rm comp}(p^*)$.

\begin{equation}
    \frac{T}{\log\left(\varepsilon^{-1}\right)} \leq C T_1(p^*)
\end{equation}

Then, $p^*$ can be found by finding the root in $]0, 1[$ of a second degree polynomial. In particular, $p^*$ is the solution of:

\begin{equation}
    p_{\rm comp}^2 = p_{\rm comm}^2 \frac{\tilde{\gamma}}{1 + \kappa r_\kappa}S_{\rm comp}^2 = (1 - p_{\rm comm})^2
\end{equation}

which leads to $p^* = \left(1 + \sqrt{\frac{\tilde{\gamma}}{1 + \kappa_{\min}}}S_{\rm comp}\right)^{-1}$.

\begin{align*}
    \frac{T}{\log\left(\varepsilon^{-1}\right)} &\leq \sqrt{2} \frac{C}{\Delta_p} \left(c_\tau \tau - 1 + \frac{1}{p^*}\right) \sqrt{\frac{1 + \kappa}{\tilde{\gamma}}} \\
    &\leq \sqrt{2} \frac{C}{\Delta_p}  \left(c_\tau \tau \sqrt{\frac{1 + \kappa}{\tilde{\gamma}}} + \frac{1}{\sqrt{ r_\kappa}}S_{\rm comp}
    \right)
\end{align*}

Finally, we use the concavity of the square root to show that:

\begin{align*}
    S_{\rm comp} &= \frac{1}{n}\sum_{i=1}^n \sum_{j=1}^{m} \sqrt{1 + L_{j}\sigma_i^{-1}}\\
    &\leq \frac{1}{n}\sum_{i=1}^n m \sqrt{\sum_{j=1}^{m} \frac{1}{m}\left(1 + L_{j}\sigma_i^{-1}\right)}\\
    &\leq \frac{1}{n}\sum_{i=1}^n m \sqrt{1 + \frac{1}{m}\kappa_i}\\
    &\leq m + \sqrt{m\kappa}
\end{align*}

Yet, this analysis only works as long as $p^* \leq 1/2$. When this constraint is not respected, we know that: $\tilde{\gamma}S_{\rm comp}^2 \leq \kappa r_\kappa$. In this case, we can simply choose $p_{\rm comp} = p_{\rm comm} = \frac{1}{2}$ and then $\rho_{\rm comm} \leq \rho_{\rm comp}$, so 

\begin{equation}
    \frac{T}{\log\left(\varepsilon^{-1}\right)} \leq \frac{C}{\Delta_p}  T_1\left(\frac{1}{2}\right) = \sqrt{2}C\left(1 + c_\tau \tau\right)\sqrt{\frac{1 + \kappa}{\tilde{\gamma}}}
\end{equation}

The sum of the two bounds is a valid upper bound in all situations, which finishes the proof.
\end{proof} 

\end{document}